\documentclass[UTF-8,reqno]{amsart}
\usepackage{amsmath, amsfonts, amsthm, amssymb, stmaryrd, color, cite, mathrsfs}

\usepackage{hyperref}
\newtheorem{theorem}{Theorem}[section]
\newtheorem{lemma}{Lemma}[section]
\newtheorem{proposition}{Proposition}[section]

\theoremstyle{definition}
\newtheorem{definition}{Definition}[section]

\theoremstyle{remark}
\newtheorem{remark}{Remark}[section]

\numberwithin{equation}{section}
\allowdisplaybreaks

\def\f{\frac}

\def\hf1{^\f{1}{1-\xi^2}}

\def\be{\begin{equation}}
\def\en{\end{equation}}
\def\bs{\begin{split}}
\def\es{\end{split}}
\def\ba{\begin{align}}
\def\ea{\end{align}}

\title[Homogenization of the distribution-dependent stochastic abstract fluid models]
{Homogenization of the distribution-dependent stochastic abstract fluid models}

\author[J. Chen]{Junlong Chen}
\address{School of Sciences, Great Bay University,  Great Bay Institute for Advanced Study, Dongguan 523000, China.
}

\address{School of Mathematics, University of Science and Technology of China, Hefei, 230026, Anhui, China.}
\email{chenjunlong@gbu.edu.cn}

\author[Z. Qiu]{Zhaoyang Qiu}
\address{School of Applied Mathematics, Nanjing University of Finance and Economics, Nanjing, 210046, China.}
\email{zhqmath@163.com}

\author[Y. Tang]{Yanbin Tang}
\address{School of Mathematics and Statistics, Hubei Key Laboratory of Engineering Modeling and Scientific Computing, Huazhong University of Science and Technology, Wuhan, Hubei 430074, China.}
\email{tangyb@hust.edu.cn}


\date{\today}

\begin{document}
\begin{abstract}
In this paper, we study the homogenization of the distribution-dependent stochastic abstract fluid models by combining the $two\!-\!scale$ convergence and martingale representative approach. A general framework of the homogenization research is established for stochastic abstract fluid models which are the type of nonlinear partial differential equations including the (distribution-dependent) stochastic Navier-Stokes equations, stochastic magneto-hydrodynamic equations, stochastic Boussinesq equations, stochastic micropolar equations, stochastic Allen-Cahn equations.
\end{abstract}

\maketitle
\section{Introduction}

Homogenization is an important subject of asymptotic analysis, which entails the process of rendering the chaotic systems more uniform or consistent in its properties, and often involves the blending of different components to create a homogeneous mixture. We could find its applications in various fields, such as the material science, engineering and physics. Particularly, in material science, it could be used to predict the homogeneous properties and improve product performance, see \cite{ngu2,moh2} for more backgrounds and applications. In this paper, we consider the qualitative homogenization of the distribution-dependent stochastic abstract fluid models with multiplicative noise
\begin{eqnarray}\label{Equ1.1}
\partial_t \mathbf{u}+A^\varepsilon \mathbf{u}+B(\mathbf{u}, \mathbf{u})=F(\mathbf{u}, \mathfrak{L}_{\mathbf{u}(t)})+G(t, \mathbf{u})\frac{dW}{dt},
\end{eqnarray}
with the boundary condition and initial data
\begin{align}\label{i1}\mathbf{u}|_{\partial D}=0,~ \mathbf{u}|_{t=0}=\mathbf{u}_0,\end{align}
where $D$ is a smooth bounded domain in $\mathbb{R}^N, ~N=2,3$. Thus, we consider the limit problem of system \eqref{Equ1.1}-\eqref{i1} as $\varepsilon\rightarrow 0$.
Such a theoretical research of this problem not only brings the convenience to numerical calculations, but also promotes the application of mathematics in dynamical and thermodynamic mechanism.

The McKean-Vlasov (distribution-dependent) stochastic differential equations were studied originally by McKean and Vlasov in the series of works \cite{mck,vla}, which could be considered as the weak limit of $M$-interacting particle systems (Propagation of Chaos). These models are more effective and powerful in studying the systems with a large number of particles, compared with studying the particle systems themselves. Therefore, the distribution-dependent stochastic differential equations are widely used in the fields of queuing systems, stochastic control and mathematical finance, see \cite{CD} for more applications.
Moreover, in the recent work \cite{MeanHuang} of  mean-field games, they introduced that the evolution of individual agents depends on the statistical distribution of other agents, and the system's behavior is captured by distribution-dependent equations.
As aforementioned, system \eqref{Equ1.1} could be interpreted as the mean field limit of the following $M$-interacting hydrodynamical systems
\begin{eqnarray*}
\partial_t \mathbf{u}^{M,i}+A^\varepsilon \mathbf{u}^{M,i}+B(\mathbf{u}^{M,i}, \mathbf{u}^{M,i})=\frac{1}{M}\sum_{j=1}^MF(\mathbf{u}^{M,i}, \mathbf{u}^{M,j})+G(t, \mathbf{u}^{M,i})\frac{dW}{dt},
\end{eqnarray*}
for $1\leq i\leq M,~ M\in \mathbb{N}^+$. A typical example of the interaction kernel $F$ is
$$F(\mathbf{u}^{M,i}, \mathbf{u}^{M,j})=\mathbf{u}^{M,i}- \mathbf{u}^{M,j}, $$
which is called the Stokes drag force, and it is proportional to the relative velocity of particles.
Corresponding to the explicit expression of $F$, in the limit system \eqref{Equ1.1}, the form of $F$ is
$$F(\mathbf{u}, \mathfrak{L}_{\mathbf{u}(t)})=\int (\mathbf{u}-y) \mathfrak{L}_{\mathbf{u}(t)}(dy),$$
where $ \mathfrak{L}_{\mathbf{u}(t)}$ is the distribution of $\mathbf{u}(t)$.

We next introduce the meaning of notations in system \eqref{Equ1.1}. The term $A^\varepsilon \mathbf{u}$ represents the diffusion effect, wherein the specific form of differential operator $A^\varepsilon$ is given by
$$A^\varepsilon=-\sum_{i,j=1}^{N}\frac{\partial}{\partial x_i}\left(a^\varepsilon_{i,j}\frac{\partial}{\partial x_j}\right),$$
where the coefficient
$$a^\varepsilon_{i,j}=a_{i,j}\left(\frac{x}{\varepsilon}, \frac{t}{\varepsilon}\right)$$
is symmetric, thus
$$a_{i,j}=a_{j,i},$$
and $a_{i,j}\in L^\infty (\mathbb{R}^N_y\times \mathbb{R}_{\tau})$, the spaces $\mathbb{R}^N_y, \mathbb{R}_{\tau}$ are the space $\mathbb{R}^N$ of variable $y=(y_1,y_2, \cdots, y_N)$, the space $\mathbb{R}$ of variable $\tau$, $\varepsilon>0$ is the scale parameter.  In material science, the coefficient $a\left(\frac{x}{\varepsilon}, \frac{t}{\varepsilon}\right)$ could be used to describe the microscopic characteristics. As the scale parameter $\varepsilon$ getting smaller, this could reveal the intrinsic properties of composite materials, thereby providing theoretical support for the efficient use of such composite materials.

We assume that operator $A^\varepsilon$ satisfies the uniformly elliptic condition, thus, there exists a constant $\kappa>0$ such that
\begin{align}\label{1.2}
\sum_{i,j=1}^Na_{i,j}(x,t)\xi_i \xi_j\geq \kappa |\xi|^2,
\end{align}
for any $x, \xi\in \mathbb{R}^N, t\in \mathbb{R}$. Here $|\cdot|$ is the Euclidean norm in $\mathbb{R}^N$. And the coefficient $a_{i,j}$ satisfies the periodicity hypothesis, for any $y\in \mathbb{R}^N, \tau\in \mathbb{R}$ and $\widetilde{y}\in \mathbb{Z}^N, \widetilde{\tau}\in \mathbb{Z}$,
\begin{align*}a_{i,j}(y+\widetilde{y}, \tau+\widetilde{\tau})=a_{i,j}(y,\tau).\end{align*}
The $B(\mathbf{u}, \mathbf{u})$ is a nonlinear term, $G$ is a noise intensity operator and $W$ is a cylindrical $Q$-Wiener process which will be introduced later.

The theory of homogenization under the periodic environment was developed by  Bensoussan et al. in the pioneering work \cite{ben}. After that, this area has received broadly attentions both in the deterministic and the random homogenization theory. Nguetseng \cite{ngu2} initially designed a general convergence result to deal with periodic rapidly oscillating coefficients, which is proved to be a powerful tool in the research of homogenization theory. Two extra variables arise in the weak limit, and capture the property of the micro oscillations. Allaire \cite{all} proposed the concept of $two\!-\!scale$ convergence based on the  Nguetseng theory and developed its properties. Then,  Allaire et al. \cite{all2,all3} further studied the homogenization of the nonlinear reaction-diffusion equation with a large oscillation reaction term, and the dispersion theory for reactive transport through porous media.  Signing \cite{sig,sig2} considered the homogenization for the unsteady Stokes type equations and the unsteady Navier-Stokes equations. Tang et al. \cite{CT} studied the homogenization of non-local nonlinear $p$-Laplacian equation with variable index and periodic structure. Using the $sigma$-convergence method, Nguetseng et al.\cite{ngu} considered the homogenization for the stationary Navier–Stokes type equations.

Regarding the random homogenization theory, some results are available related to stochastic PDEs with periodic coefficients, see \cite{ ich, par, wang1, wang2} and references therein. Particularly,  Sango et al.\cite{raz} generalized the result \cite{all2} to the stochastic case under the condition of the almost periodic framework, moreover,  in \cite{moh,san} the authors extended the homogenization result to the stochastic semilinear reaction-diffusion equations with non-Lipschitz forcing, and to the linear hyperbolic type of stochastic PDEs, see also \cite{ben2,raz2,san3} for the related results. Note that these works only involve the nonlinear case of parabolic type equations with the $m$-dimensional standard Wiener process, thus, $W=(W_1,\cdots, W_m)\in C([0,T];\mathbb{R}^m)$, while there is no result in the genuine-nonlinear stochastic PDEs, such as the  hydrodynamical systems, the Allen-Cahn equations.  Duan et al. \cite{HJS2} considered the homogenization of non-local stochastic PDEs related to stochastic differential equations with L\'{e}vy noise, see also \cite{HJS} for the non-symmetric jump processes.

We consider the distribution-dependent stochastic abstract fluid models in this paper, a general framework of the homogenization research is established for the genuine-nonlinear fluid models which cover the stochastic Navier-Stokes equations, stochastic micropolar equations, stochastic magneto-hydrodynamic equations, etc.. The main contribution of this paper is that we improve the previous results in the following three points: (i). We improve the result to a cylindrical $Q$-Wiener process, thus, an infinite-dimensional process. This  is more general
than what was done in the previous works  \cite{moh,san, raz, ben2,raz2,san3} which concerned stochastic linear or semilinear PDEs with the finite-dimensional noise $W$;
(ii) The systems we confront include the strong nonlinear term. We only impose the locally monotonicity condition on $B(\mathbf{u}, \mathbf{u})$, while the previous works dealt with the nonlinear force $f(t, \mathbf{u})$ satisfying the global Lipschitz and linear growth conditions;
(iii) The third  point is that the external force $F$ relies on $\mathfrak{L}_{\mathbf{u}(t)}$ the distribution of solution. To the best of our knowledge, this is the first result in the literature on the research of homogenization. Also, we only impose the locally monotonicity condition on $F$ which covers the case of polynomial function $x-x^3$.

The compactness spaces needed for passing the limit are endowed with the weak topology, the classical Skorokhod representation theorem does not appropriate for our case. We use the Jakubowski version of the Skorokhod representation theorem to achieve it by verifying the existing of a countable set of continuous real-valued
functions separating points of compactness spaces. The infinite-dimensional noise and the dependence of distribution cause the difficulty in passing the limit, the martingale representative approach is invoked to  overcome it.

It is worthwhile to notice that $a^\varepsilon$ could be more general, that is to say $a^\varepsilon_{i,j}=a_{i,j}\left(\frac{x}{\varepsilon}, \frac{t}{\varepsilon_1}\right)$, where $\varepsilon_1$ is a positive function of $\varepsilon$ tending to $0$ as $\varepsilon\rightarrow 0$. We believe that  multi-scale $sigma$-convergence could be established.  Another interested question is that will we get the same limit equations as $\varepsilon\rightarrow 0$ firstly, and then $M\rightarrow\infty$ in the $M$-interacting hydrodynamical systems? Such a study is postponed till further publication. Also, the periodicity hypothesis could be extended to the almost periodic case, like \cite{raz, raz2}.

The rest of paper is organized as follows. We introduce some preliminaries including functional spaces and operators, assumptions, and the main result in section 2. At the end of this section, we illustrate our result to several typical  hydrodynamical systems and the Allen-Cahn equations. In section 3, we establish the necessary a priori estimates and the compactness. In section 4, we recall the $two\!-\!scale$ convergence and  prove the homogenization result. We obtain a corrector result which improves the convergence of $\nabla \mathbf{u}^\varepsilon$ in $L^2(\Omega\times D_t),~\Sigma\!-\!{\rm weak}$ to the $L^2(\Omega\times D_t),~\Sigma\!-\!{\rm strong}$ in section 5.
An appendix is included afterwards to state two results that are frequently used in the paper. Throughout the paper, universal constants depending only on the dimension, or initial data and other parameters are denoted by $C$, etc. which may be different at each occurrence.

\section{preliminary}
In this section, we introduce some preliminaries including functional spaces and operators, assumptions, and the main result. Several typical examples of  hydrodynamical systems and the Allen-Cahn equations that could be covered by our theory are illustrated at the end of this part.

{\bf Functional spaces and operators}. For any $m\in \mathbb{N}^+, p\geq 1$, denote by $H^{m, p}$ the Sobolev spaces of functions having distributional derivatives up to order $m\in \mathbb{N}^{+}$, which is integrable in $L^{p}(D)$. We denote by $H^{-m, p'}$ the dual of $H^{m, p}$, $p'$ is the conjugate index of $p$.
 Denote by $H:=L^2, V:=H^{1,2}$ two Hilbert spaces with $V\subset H$, which is dense and compact. Denote by $V'$ the dual of space $V$, then we could define
the Gelfand inclusions $V\subset H\subset V'$. We denote by $(\cdot, \cdot), \|\cdot\|_{H}, \|\cdot\|_{V}$ the inner product of $H$ and the norms $H, V$. The duality product between $V, V'$ is denoted by $(\cdot, \cdot)_{V\times V'}$.

Let $\mathscr{P}(X)$ be the space of all probability measures on Banach space $X$ equipped with the weak topology, for any $p>0$, denote
$$\mathscr{P}_p(X):=\left\{\mu\in \mathscr{P}(X): \mu(\|\cdot\|_{X}^p)=\int_{X}\|x\|^p_{X}\mu(dx)<\infty\right\}$$
endowed with the $p$-Wasserstein distance
$$\mathcal{W}_{p,X}:=\inf_{\pi\in \mathscr{C}(\mu,\nu)}\left(\int_{X\times X}\|x-y\|_{X}^p\pi(dx, dy)\right)^{\frac{1}{p\vee 1}},$$
where $\mu,\nu\in \mathscr{P}(X)$ and $\mathscr{C}(\mu,\nu)$ is the set of all couples $\mu$ and $\nu$. Note that the space  $\mathscr{P}_p(X)$ is a Polish space under the $p$-Wasserstein distance, see \cite{CD} for more details.

Denote by $C([0,T];X)$ the space of all the continuous functions from $[0,T]$ into Banach space $X$, with the norm
$$\|\mathbf{u}\|_{C([0,T];X)}=\sup_{t\in [0,T]}\|\mathbf{u}(t)\|_{X}.$$
The   $p$-Wasserstein distance on $\mathscr{P}_{p}(C([0,T];X))$ is given by
$$\mathcal{W}_{p,T,X}=\inf_{\pi\in \mathscr{C}(\mu,\nu)}\left(\int_{C([0,T];X)\times C([0,T];X)}\|x-y\|_{C([0,T];X)}^p\pi(dx, dy)\right)^{\frac{1}{p\vee 1}}.$$

We assume that $B$ is a bilinear map from $V\times V$ into $H^{-1, \frac{N}{N-1}}$ with following properties
\begin{align}\label{2.1*}
(B(\mathbf{u},\mathbf{v}), \mathbf{v})=0, ~{\rm for~ any}~\mathbf{u}, \mathbf{v}\in V,
\end{align}
and
\begin{align}\label{2.3}\|B(\mathbf{u}, \mathbf{u})\|_{H^{-1, \frac{N}{N-1}}}^2\leq C\|\mathbf{u}\|_{H}^2\|\mathbf{u}\|_{V}^2,\end{align}
where the positive constant $C$ relies only on dimension $N$ and the domain $D$.
Particularly, for $N=2$, we have
\begin{align*}\|B(\mathbf{u}, \mathbf{u})\|_{V'}^2\leq C\|\mathbf{u}\|_{H}^2\|\mathbf{u}\|_{V}^2,\end{align*}
thus,
\begin{align}\label{2.2}
(B(\mathbf{u},\mathbf{v}), \mathbf{u})\leq C\|\mathbf{u}\|_{H}\|\mathbf{u}\|_{V}\|\mathbf{v}\|_{V}, ~{\rm for~ any}~\mathbf{u}, \mathbf{v}\in V.
\end{align}
Using \eqref{2.1*} and \eqref{2.2}, we see that the operator $B$ satisfies the local monotonicity condition
\begin{align}\label{2.4}
(B(\mathbf{u}, \mathbf{u})-B(\mathbf{v}, \mathbf{v}), \mathbf{u}-\mathbf{v})=(B(\mathbf{u}-\mathbf{v}, \mathbf{u}), \mathbf{u}-\mathbf{v})\leq C\|\mathbf{u}-\mathbf{v}\|_{H}\|\mathbf{v}\|_{V}\|\mathbf{u}-\mathbf{v}\|_{V}.
\end{align}

For the linear diffusion term,  we understand $A^\varepsilon$ as the bounded operator from $V$ into $V'$ with the duality product
$$(A^\varepsilon \mathbf{u}, \mathbf{v})_{V'\times V}=\sum_{i,j=1}^{N}\left(a^\varepsilon_{i,j}\frac{\partial \mathbf{u}}{\partial x_j}, \frac{\partial \mathbf{v}}{\partial x_i}\right), ~{\rm for}~\mathbf{u}, \mathbf{v}\in V.$$
Since the embedding $V$ into $H$ is compact, it follows that $(A^\varepsilon)^{-1}$ as a map from $H$ into $V$ is compact on $H$. From the symmetrically and the compactness of operator, we have the existence of a complete orthonormal basis $\{\mathbf{e}_{k}\}_{k\geq 1}$ for $H$ of eigenfunctions of $A^\varepsilon$.

{\bf Stochastic framework}. Let $\mathcal{S}:=(\Omega,\mathcal{F},\{\mathcal{F}_{t}\}_{t\geq0},\mathrm{P}, W)$ be a fixed stochastic basis and $(\Omega,\mathcal{F},\mathrm{P})$ a complete probability space. $\{\mathcal{F}_{t}\}_{t\geq0}$ is a filtration satisfying all usual conditions.
Denote by $L^p(\Omega; L^q(0,T;X)), p\in [1,\infty), q\in [1, \infty]$ the space of processes with values in $X$ defined on $\Omega\times [0,T]$ such that

i. $\mathbf{u}$ is measurable with respect to $(\omega, t)$, and for each $t\geq 0$, $\mathbf{u}(t)$ is $\mathcal{F}_{t}$-measurable;

ii. For almost all $(\omega, t)\in \Omega\times [0,T]$, $\mathbf{u}\in X$ and
\begin{align*}
\|\mathbf{u}\|^p_{L^p(\Omega; L^q(0,T;X))}\!=\!
\begin{cases}
\mathrm{E}\left(\int_{0}^{T}\|\mathbf{u}\|_{X}^qdt\right)^\frac{p}{q},~& {\rm if}~q\in [1,\infty),\\
\mathrm{E}\left(\sup_{t\in[0,T]}\|\mathbf{u}\|^p_X\right),~ & {\rm if}~q=\infty.
\end{cases}
\end{align*}
Here, $\mathrm{E}$ denotes the mathematical expectation.

Assume that $Q$ is a linear positive operator on the Hilbert space $H$, which is trace and hence compact. Let $W$ be a Wiener process defined on the Hilbert space $H$ with covariance operator $Q$, which is adapted to the complete, right continuous filtration $\{\mathcal{F}_{t}\}_{t\geq 0}$. Let $\{\mathbf{e}_{k}\}_{k\geq 1}$ be a complete orthonormal basis of $H$ such that $Q\mathbf{e}_{i}=\lambda_{i}\mathbf{e}_{i}$, then $W$ can be written formally as the expansion $W(t,\omega)=\sum_{k\geq 1}\sqrt{\lambda_{k}}\mathbf{e}_{k}W_k(t,\omega)$, where $\{W_{k}\}_{k\geq 1}$ is a sequence of independent standard 1-D Brownian motions, see \cite{Zabczyk} for more details.

Let $H_{0}=Q^{\frac{1}{2}}H$, then $H_{0}$ is a Hilbert space with the inner product
\begin{equation*}
\langle h,g\rangle_{H_{0}}=\langle Q^{-\frac{1}{2}}h,Q^{-\frac{1}{2}}g\rangle_{H},~ \forall~~ h,g\in H_{0},
\end{equation*}
with the induced norm $\|\cdot\|_{H_{0}}^{2}=\langle\cdot,\cdot\rangle_{H_{0}}$. The imbedding map $i:H_{0}\rightarrow H$ is Hilbert-Schmidt and hence compact operator with $ii^{\ast}=Q$, where $i^{\ast}$ is the adjoint of the operator $i$. We have that, $W\in C([0,T]; H_0)$ almost surely. Now let us consider another separable Hilbert space $X$ and let $L_{Q}(H_{0},X)$ be the space of linear operators $S:H_{0}\rightarrow X$ such that $SQ^{\frac{1}{2}}$ is a linear Hilbert-Schmidt
operator from $H$ to $X$,  endowed with
  the norm $$\|S\|_{L_{Q}}^{2}=tr(SQS^{\ast}) =\sum\limits_{k\geq 1}\| SQ^{\frac{1}{2}}\mathbf{e}_{k}\|_{X}^{2}.$$
 Set $$L_2(H,X)=\left\{SQ^{\frac{1}{2}}: \, S\in L_{Q}(H_{0},X)\right\},$$ the norm is defined by
  $\|f\|^2_{L_2(H,X)}  =\sum\limits_{k\geq 1}\|f\mathbf{e}_{k}\|_{X}^{2}$.

In order to control the martingale part, we recall the following well-known Burkholder-Davis-Gundy inequality: for any $f\in  L^{p}(\Omega;L^{2}_{loc}([0,\infty),L_{2}(H,X)))$,  by taking $f_{k}=f \mathbf{e}_{k}$, there exists a constant $c_p>0$ such that
\begin{eqnarray*}
\mathrm{E}\left(\sup_{t\in [0,T]}\left\|\int_{0}^{t}f d\mathcal{W}\right\|_{X}^{p}\right)\leq c_{p}\mathrm{E}\left(\int_{0}^{T}\|f\|_{L_{2}(H,X)}^{2}dt\right)^{\frac{p}{2}}
=c_{p}\mathrm{E}\left(\int_{0}^{T}\sum_{k\geq 1}\|f_k\|_{X}^{2}dt\right)^{\frac{p}{2}},
\end{eqnarray*}
for any $p\in[1,\infty)$, see also \cite[Theorem 4.36]{Zabczyk}.

{\bf Assumptions on $F$ and $G$}.   Assume that there exists a positive constant $C$ such that the measurable map
$$F: H\times \mathscr{P}(V')\rightarrow H$$
satisfies

$(\mathbf{F.1})$ [Growth condition] for any $\mathbf{u}\in H$, $\mu\in \mathscr{P}_2(H)$,
 \begin{align*}
(F(\mathbf{u}, \mu), \mathbf{u})\leq C(\|\mathbf{u}\|_H^2+\mu(\|\cdot\|_{H}^2)-\|\mathbf{u}\|_{L^{4}}^{4});
\end{align*}

 $(\mathbf{F.2})$ for any $\mathbf{v}\in H$, and $\mathbf{u}^\varepsilon, \mathbf{u}\in H$ with $\|\mathbf{u}^\varepsilon-\mathbf{u}\|_{H}\rightarrow 0$, $\mu^\varepsilon, \mu\in \mathscr{P}_2(H)$ with $\mu^\varepsilon\rightarrow\mu$ in $\mathscr{P}_2(V')$,
 $$\lim_{\varepsilon\rightarrow 0}(F(\mathbf{u}^\varepsilon, \mu^\varepsilon), \mathbf{v})=(F(\mathbf{u}, \mu), \mathbf{v});$$

$(\mathbf{F.3})$ [Local monotonicity condition]  for any $\mathbf{u}_1, \mathbf{u}_2\in H$, $a\geq 2$, $\mu_1, \mu_2\in \mathscr{P}_a(H)$, the locally Lipschitz condition holds
 $$(F(\mathbf{u}_1, \mu_1)-F(\mathbf{u}_2, \mu_2), \mathbf{u}_1-\mathbf{u}_2)\nonumber$$
 $$\leq C(1+\mu_1(\|\cdot\|_{H}^a)+\mu_2(\|\cdot\|_{H}^a))\|\mathbf{u}_1-\mathbf{u}_2\|_{H}^2
 +\eta\|\mathbf{u}_1-\mathbf{u}_2\|_{V}^2
 +\ell \mathcal{W}_{2, H}(\mu_1, \mu_2),$$
where  $\ell$ is a certain constant being smaller than $\frac{\kappa-2\eta}{2}$ and $\eta$ is a constant satisfying $\eta<\frac{\kappa}{2}$.

An example of the measurable map $F$ is
$$F(\mathbf{u}, \mathfrak{L}_{\mathbf{u}(t)})=F_1(\mathbf{u}, \mathfrak{L}_{\mathbf{u}(t)})+F_2(t, \mathbf{u}),$$
where $F_1(\mathbf{u}, \mathfrak{L}_{\mathbf{u}(t)})=\int (\mathbf{u}-y) \mathfrak{L}_{\mathbf{u}(t)}(dy)$, while $F_2(t, \mathbf{u})$ satisfies
$$F_2(t, 0)=0,$$
and
\begin{align*}
(F_2(t, \mathbf{u}), \mathbf{u})\leq C(-\|\mathbf{u}\|_{L^{4}}^{4}+\|\mathbf{u}\|_{H}^{2}),
\end{align*}
and
 \begin{align*}&(F_2(t, \mathbf{u}_1)-F_2(t, \mathbf{u}_2), \mathbf{u}_1-\mathbf{u}_2)\leq C\|\mathbf{u}_1-\mathbf{u}_2\|_{H}^2+\eta\|\mathbf{u}_1-\mathbf{u}_2\|_{V}^2.
\end{align*}
Obviously, map $F$ satisfies assumptions $(\mathbf{F.1}), (\mathbf{F.3})$.

A typical example of function $F_2(t, \mathbf{u})$ satisfying above conditions is a cubic polynomial, thus,
$$F_2(t, \mathbf{u})=-c(t)\mathbf{u}^{3}+\xi(t)\mathbf{u},$$
where $c(t)>c>0$, $c$ is a constant, $\xi(t)>0$ is continuous function of $t$. Therefore, our result could also cover the stochastic Allen-Cahn equations.

 For the noise intensity operator $G$, we impose the following conditions: There exists a positive constant $C$ such that the Lipschitz and linear growth conditions hold

 $(\mathbf{G.1})~~ \|G(t, \mathbf{u}_1)-G(t, \mathbf{u}_2)\|^2_{L_2(H,H)}\leq C\|\mathbf{u}_1-\mathbf{u}_2\|_{H}^2, ~{\rm for}~ \mathbf{u}_1, \mathbf{u}_2\in H; $

 $(\mathbf{G.2})~~ \|G(t, \mathbf{u})\|^2_{L_2(H,H)}\leq C(1+\|\mathbf{u}\|_{H}^2),  ~{\rm for}~ \mathbf{u}\in H.$

When $F_2(t, \mathbf{u})=0$, the well-posedness result of system \eqref{Equ1.1}-\eqref{i1} is given by \cite{HHL}. Here we also have the following well-posedness result.
\begin{lemma} Let $\mathcal{S}$ be the fixed stochastic basis. Under the assumptions $(\mathbf{F.i}), \mathbf{i}=1,2,3$, $(\mathbf{G.j}), \mathbf{j}=1,2$ and initial data $\mathbf{u}_0\in H$, for every $\varepsilon>0$, $T>0$,  there exists a pathwise solution of system \eqref{Equ1.1}-\eqref{i1} in the following sense:

i. $\mathbf{u}^\varepsilon$ is a sequence of $H$-valued $\mathcal{F}_t$-adapt processes with the regularity
$$\mathbf{u}^\varepsilon\in L^p(\Omega; L^\infty(0,T;H)\cap L^2(0,T;V)\cap L^{4}(D\times (0,T)))$$
for any $p\geq 2$;

ii. for any $t\in [0,T]$, $\varphi\in V$, it holds $\mathrm{P}$ a.s.
\begin{align}
(\mathbf{u}^\varepsilon(t), \varphi)=&(\mathbf{u}_0, \varphi)-\int_0^t(A^\varepsilon \mathbf{u}^\varepsilon, \varphi)_{ V'\times V}ds-\int_0^t(B(\mathbf{u}^\varepsilon, \mathbf{u}^\varepsilon),\varphi)ds\nonumber\\
&+\int_{0}^{t} (F(\mathbf{u}^\varepsilon, \mathfrak{L}_{\mathbf{u}^\varepsilon(s)}), \varphi)ds+\int_{0}^{t} (G(s, \mathbf{u}^\varepsilon)dW, \varphi);
\end{align}

iii. $\mathbf{u}\in C([0,T];V')$, $\mathrm{P}$ a.s.

If $N=2$, we have that the solution is pathwise unique. Particularly, the pathwise uniqueness of solution to the stochastic Allen-Cahn equations holds for both $N=2,3$.
\end{lemma}

{\bf Main result}. We formulate the main result of our paper.
\begin{theorem} if the assumptions $(\mathbf{F.i}), \mathbf{i}=1,2,3$, $(\mathbf{G.j}), \mathbf{j}=1,2$ hold and assume that $\mathbf{u}^\varepsilon$ is a sequence of solutions of system \eqref{Equ1.1}-\eqref{i1} with regularity $$\mathbf{u}^\varepsilon\in L^p(\Omega; L^\infty(0,T;H)\cap L^2(0,T; V)\cap L^{4}(D\times (0,T))),~ uniformly ~in~ \varepsilon,~ p\geq 2.$$
Then,  we have $ \mathbf{u}^\varepsilon\rightarrow \mathbf{u}$
strongly in $L^2(\Omega; L^2(D\times (0,T)))$ as $\varepsilon\rightarrow 0$, and the limit process $\mathbf{u}$ shares the same regularity with $\mathbf{u}^\varepsilon$,
which satisfies the following homogenization equations
\begin{eqnarray*}
\partial_t \mathbf{u}+\widetilde{A }\mathbf{u}+B(\mathbf{u}, \mathbf{u})=F(\mathbf{u}, \mathfrak{L}_{\mathbf{u}(t)})+G(t, \mathbf{u})\frac{dW}{dt},
\end{eqnarray*}
in which $\widetilde{A}=(\widetilde{A}_{kl})_{k,l=1,\cdots, N}$ is the so-called homogenized operator associated to $A^\varepsilon$, $\widetilde{A}_{kl}$ is given by
$$\widetilde{A}_{kl}=-\sum_{i,j=1}^N\mathbf{a}_{i,j,k,l}\frac{\partial^2}{\partial x_i \partial y_j},~1\leq k,l\leq  N,$$
and the coefficient
$$\mathbf{a}_{i,j,k,l}:=\int_{D_\tau}a_{i,j}dy d\tau+\int_{D_{\tau}}a_{i,j}\frac{\partial\eta^l_{i,k}}{\partial y_j}dy d\tau,~1\leq i,j,k,l\leq  N,$$
and $\eta_{i,k}$ is the solution of variational problem \eqref{var1}.
\end{theorem}

 {\bf Example 1: Distribution-dependent stochastic Navier-Stokes equations}

 We first apply the abstract framework to the distribution-dependent stochastic Navier-Stokes equations:
\begin{eqnarray}\label{Equ1.1*}
\left\{\begin{array}{ll}
\!\!\!\partial_t \mathbf{u}+A^\varepsilon \mathbf{u}+\mathbf{u}\cdot\nabla\mathbf{u}+\nabla p=F(\mathbf{u}, \mathfrak{L}_{\mathbf{u}(t)})+G(t, \mathbf{u})\frac{dW}{dt},\\
\!\!\!\nabla\cdot \mathbf{u}=0,\\
\!\!\! \mathbf{u}(0)=\mathbf{u}_0,
\end{array}\right.
\end{eqnarray}
where $\mathbf{u}$ is the velocity field, $p$ is the pressure, $W$ is a cylindrical $Q$-Wiener process. The domain $D\subset \mathbb{R}^N$ we consider is smooth bounded and $\mathbf{u}$ satisfies the non-slip boundary condition on $\partial D$, thus,
$$\mathbf{u}|_{\partial D}=0.$$

 Denote space
$$H:=\left\{\mathbf{u}\in L^2: \nabla\cdot\mathbf{u}=0, \mathbf{u}|_{\partial D}=0\right\}$$
with the norm
$$\|\mathbf{u}\|^2_{H}=\|\mathbf{u}\|^2_{L^2}=(\mathbf{u},\mathbf{u}),$$
where $(\cdot, \cdot)$ means the inner product of $L^2$.

Denote
$$V:=\left\{\mathbf{u}\in H: \frac{\partial \mathbf{u}}{\partial x}\in L^2\right\}$$
with the norm
$$\|\mathbf{u}\|^2_{V}=\|\nabla \mathbf{u}\|^2_{L^2}=\sum_{i,j=1}^N\int_{D}\left|\frac{\partial \mathbf{u}_j}{\partial x_i}\right|^2dx.$$

Define the bilinear map
$$B(\mathbf{u},\mathbf{v})=P(\mathbf{u}\cdot\nabla) \mathbf{v},$$
where $P$ is the Helmhotz–Hodge projection operator from $L^2$ into $H$, and define the tri-linear map $b(\cdot, \cdot, \cdot): V\times V\times V\rightarrow \mathbb{R}$ by
$$b(\mathbf{u},\mathbf{v},\mathbf{w})=\int_{D}\mathbf{u}\cdot\nabla \mathbf{v}\cdot \mathbf{w}dx, ~ \mathbf{u},\mathbf{v},\mathbf{w}\in V,$$
then,
$$(B(\mathbf{u},\mathbf{v}), \mathbf{w})=b(\mathbf{u},\mathbf{v},\mathbf{w}).$$
Integrating by parts, we obtain
$$(B(\mathbf{u},\mathbf{v}), \mathbf{w})=-(B(\mathbf{u},\mathbf{w}), \mathbf{v}), ~\mathbf{u}, \mathbf{v}, \mathbf{w}\in V,$$
which implies
\begin{align*}\label{2.1}b(\mathbf{u},\mathbf{v},\mathbf{v})=0, ~\mathbf{u},\mathbf{v} \in V.\end{align*}
Moreover, we could easily check that $B(\mathbf{u, \mathbf{v}})$ satisfies \eqref{2.3}. Indeed, for $N=3$ and any $\phi\in H^{1,3}$, using the embedding $V\hookrightarrow L^6$ we see
$$(B(\mathbf{u},\mathbf{u}), \phi)=-(B(\mathbf{u},\phi), \mathbf{u})$$
$$\leq \|\mathbf{u}\|_{H}\|\mathbf{u}\|_{L^6}\|\nabla \phi\|_{L^3}\leq C\|\phi\|_{H^{1,3}}\|\mathbf{u}\|_{H}\|\mathbf{u}\|_{V},
$$
and for $N=2$ and any $\phi\in V$, by $\|\mathbf{u}\|^2_{L^4}\leq C\|\mathbf{u}\|_{L^2}\|\mathbf{u}\|_{V}$ we see
$$(B(\mathbf{u},\mathbf{u}), \phi)=-(B(\mathbf{u},\phi), \mathbf{u})\leq \|\mathbf{u}\|^2_{L^4}\|\phi\|_{V}\leq C\| \phi\|_{V}\|\mathbf{u}\|_{H}\|\mathbf{u}\|_{V},
$$
We obtain the desired estimates \eqref{2.3} and \eqref{2.2}.

Denote by $F^P=PF, G^P=PG$,  after an application of operator $P$ to equations \eqref{Equ1.1*}, we could rewrite it by abstract form as an evolutionary dynamical system
\begin{equation*}
d\mathbf{u}+A^\varepsilon\mathbf{u} dt+B(\mathbf{u},\mathbf{u})dt=F^P(\mathbf{u}, \mathfrak{L}_{\mathbf{u}(t)})dt+G^P(t, \mathbf{u})dW,
\end{equation*}
with initial data $\mathbf{u}(0)=\mathbf{u}_0$.

 {\bf Example 2: Distribution-dependent stochastic magneto-hydrodynamic equations}

The next example which could be covered by our framework is the distribution-dependent stochastic magneto-hydrodynamic equations
\begin{equation}\label{2.11}
\left\{
  \begin{array}{ll}
   d\mathbf{u}+\left(\mathbf{u}\cdot\nabla\mathbf{u}-\mathbf{T}\cdot\nabla\mathbf{T}+A_1^\varepsilon \mathbf{u}+\nabla \left(p+\frac{|\mathbf{T}|^2}{2}\right)\right)dt=F(\mathbf{u}, \mathfrak{L}_{\mathbf{u}(t)})dt+G_1(\mathbf{u}, \mathbf{T})dW_1,\\
     d\mathbf{T}+(\mathbf{u}\cdot\nabla\mathbf{T}-\mathbf{T}\cdot \nabla \mathbf{u}+A_2^\varepsilon \mathbf{T})dt=G_2(\mathbf{u}, \mathbf{T})dW_2,\\
     \nabla\cdot \mathbf{u}=0,~\nabla \cdot \mathbf{T}=0,\\
  \end{array}
\right.
\end{equation}
with initial data $\mathbf{u}(0)=\mathbf{u}_{0}, \mathbf{T}(0)=\mathbf{T}_{0}$, where $\mathbf{u}$ is the velocity field, $\mathbf{T}$ is the magnetic field, $p$ is pressure, $W_1, W_2$ are two i.i.d. Wiener processes.
Recent literature \cite{lanzhou} has explored homogenization of the 3D Hall magneto-hydrodynamic equations in the absence of random term and $F$. 

 The evolution also occurs in domain $D$
 which is a smooth bounded domain in $\mathbb{R}^N$. Two functions $\mathbf{u}, \mathbf{T}$ satisfy the non-slip boundary condition and Neumann boundary condition on $\partial D$ respectively, thus,
$$\mathbf{u}|_{\partial D}=0,~~ \frac{\partial\mathbf{T}}{\partial \vec{\mathbf{n}}}\big|_{\partial D}=0,$$
$\vec{\mathbf{n}}$ is the outward normal to $\partial D$.

   Using the divergence-free condition, applying the Helmhotz–Hodge projection operator $P$ on both sides of \eqref{2.11}, we rewrite the equation as follows
\begin{equation}\label{2.15*}
\left\{
  \begin{array}{ll}
   d\mathbf{u}+P(\mathbf{u}\cdot\nabla\mathbf{u}-\mathbf{T}\cdot\nabla\mathbf{T}-A^\varepsilon \mathbf{u})dt=PF(\mathbf{u}, \mathfrak{L}_{\mathbf{u}(t)})dt+PG_1(\mathbf{u}, \mathbf{T})dW_1,\\
     d\mathbf{T}+P(\mathbf{u}\cdot\nabla\mathbf{T}-\mathbf{T}\cdot \nabla \mathbf{u}-A^\varepsilon \mathbf{T})dt=PG_2(\mathbf{u}, \mathbf{T})dW_2.
  \end{array}
\right.
\end{equation}

Let us introduce the spaces
\begin{eqnarray*}
&&H:=\left\{(\mathbf{u}, \mathbf{T}) \in (L^2)^2: \nabla\cdot \mathbf{u}=0,~\nabla \cdot \mathbf{T}=0,~ \mathbf{u}|_{\partial D}=0,~~ \frac{\partial\mathbf{T}}{\partial \vec{\mathbf{n}}}\bigg|_{\partial D}=0\right\},\\
&&V:=\left\{(\mathbf{u}, \mathbf{T})\in H: \frac{\partial \mathbf{u}}{\partial x}\in L^2, \frac{\partial \mathbf{T}}{\partial x}\in L^2\right\},
\end{eqnarray*}
with the norm
$$\|\mathbf{u}, \mathbf{T}\|^2_{H}=\|\mathbf{u}\|^2_{L^2}+\|\mathbf{T}\|^2_{L^2}, ~\|\mathbf{u}, \mathbf{T}\|^2_{V}=\|\nabla \mathbf{u}\|^2_{L^2}+\|\nabla \mathbf{T}\|^2_{L^2}.$$

Define by $A^\varepsilon$  the operator
$$
A^\varepsilon X=P\left(
\begin{array}{ccc}
A_1^\varepsilon \mathbf{u}\\
A_2^\varepsilon\mathbf{T}
\end{array}
\right),
$$
 and by $B(X, \widetilde{X})=B_1(X, \widetilde{X})+B_2(X, \widetilde{X})$, where
$$
B_1(X, \widetilde{X})=P\left(
\begin{array}{ccc}
\mathbf{u}\cdot \nabla \widetilde{\mathbf{u}}\\
\mathbf{u}\cdot \nabla \widetilde{\mathbf{T}}
\end{array}
\right),
$$
and
$$
B_2(X, \widetilde{X})=P\left(
\begin{array}{ccc}
-\mathbf{T}\cdot\nabla\widetilde{\mathbf{T}}\\
-\mathbf{T}\cdot \nabla \widetilde{\mathbf{u}}
\end{array}
\right),
$$
for $X=(\mathbf{u}, \mathbf{T})$, $\widetilde{X}=(\widetilde{\mathbf{u}}, \widetilde{\mathbf{T}})$.
Denote by
$$
F^P=P\left(
\begin{array}{ccc}
F\\
0
\end{array}
\right),~~
G^P=P\left(
\begin{array}{ccc}
G_1& 0\\
0& G_2
\end{array}
\right), ~~
W=\left(\begin{array}{ccc}
W_1\\
W_2
\end{array}
\right).
$$
For the operator $B$, we may show that it is well-defined,  and easily check that properties \eqref{2.1*}-\eqref{2.2} are satisfied.

With the above definitions in place we may write system \eqref{2.15*} supplemented by the boundary conditions in the abstract form
\begin{equation*}
dX+A^\varepsilon Xdt+B(X,X)dt=F^P(X, \mathfrak{L}_{X})dt+G^P(X)dW,
\end{equation*}
with initial data $X(0)=X_0$.
\begin{remark} Actually, more concrete  hydrodynamical systems with distribution-dependent force, such as the stochastic Boussinesq equations, stochastic micropolar equations, could be transformed into above abstract models. Therefore, in the followings, we keep the form $B(\mathbf{u}, \mathbf{u})=\mathbf{u}\cdot \nabla\mathbf{u}$ in mind.
\end{remark}

{\bf Example 3: Distribution-dependent stochastic Allen-Cahn equations}

The third example we illustrate is the distribution-dependent stochastic Allen-Cahn equations which is a type of reaction-diffusion equations
\begin{eqnarray}\label{2.9*}
\left\{\begin{array}{ll}
\!\!\!\partial_t \mathbf{u}+A^\varepsilon \mathbf{u}=F(\mathbf{u}, \mathfrak{L}_{\mathbf{u}(t)})-\mathbf{u}^3+\mathbf{u}+G(t, \mathbf{u})\frac{dW}{dt},\\
\!\!\!\mathbf{u}|_{\partial D}=0,\\
\!\!\! \mathbf{u}(0)=\mathbf{u}_0,
\end{array}\right.
\end{eqnarray}
 modelling the effect of thermal perturbations and simulating rare events in infinite-dimensional stochastic systems.
 The homogenization of  deterministic Allen-Cahn  equations has been studied in \cite{Allen2022}. Here $B(\mathbf{u},\mathbf{u})\equiv 0$, $W$ is a cylindrical $Q$-Wiener process, we only need to verify that $F_2(t, \mathbf{u})=-\mathbf{u}^3+\mathbf{u}$ satisfies the conditions we imposed. First, we see $F_2(t, 0)=0$ and \begin{align*}
(F_2(t, \mathbf{u}), \mathbf{u})\leq C(-\|\mathbf{u}\|_{L^{4}}^{4}+\|\mathbf{u}\|_{H}^{2}).
\end{align*}
Moreover, we obtain
 \begin{align*}&(F_2(t, \mathbf{u}_1)-F_2(t, \mathbf{u}_2), \mathbf{u}_1-\mathbf{u}_2)\nonumber\\
& =-(\mathbf{u}^3_1-\mathbf{u}^3_2, \mathbf{u}_1-\mathbf{u}_2)+\|\mathbf{u}_1-\mathbf{u}_2\|_{H}^2\leq \|\mathbf{u}_1-\mathbf{u}_2\|_{H}^2.
\end{align*}
We next show that for any $\varphi\in H$, $\lim\limits_{\varepsilon\rightarrow 0}(F_2(t, \mathbf{u}^\varepsilon), \varphi)=(F_2(t, \mathbf{u}), \varphi)$, $\mathrm{P}$ a.s. as $\mathbf{u}^\varepsilon\rightarrow \mathbf{u}$ in $H$. Choosing $\varphi\in C^1$ firstly, by the H\"{o}lder inequality we get
\begin{align*}
&(F_2(t, \mathbf{u}^\varepsilon), \varphi)-(F_2(t, \mathbf{u}), \varphi)\nonumber\\
&\leq (\mathbf{u}^\varepsilon-\mathbf{u}, \varphi)-((\mathbf{u}^\varepsilon-\mathbf{u})((\mathbf{u}^\varepsilon)^2+\mathbf{u}^\varepsilon\mathbf{u}+\mathbf{u}^2), \varphi)\nonumber\\
&\leq (\mathbf{u}^\varepsilon-\mathbf{u}, \varphi)+C\|\varphi\|_{C^1}(\|\mathbf{u}^\varepsilon\|_{L^{4}}^{2}+\|\mathbf{u}\|_{L^{4}}^{2})\|\mathbf{u}^\varepsilon-\mathbf{u}\|_{H},
\end{align*}
then we obtain the desired convergence by $\mathbf{u}^\varepsilon\rightarrow \mathbf{u}$ in $H$ and the density of $C^1$ in $H$. Hence, our framework could be used to the Allen-Cahn equations \eqref{2.9*}.

\section{The stochastic compactness} In this section, we mainly focus on the proof of the tightness of a family of probability measures.

{\bf The a priori estimates}. We begin with establishing the necessary a priori estimates.

\begin{lemma}\label{lem3.1} The sequence of solutions $\mathbf{u}^\varepsilon$ has the following uniform estimates of $\varepsilon$
\begin{align}\label{3.1}
\mathrm{E}\left(\sup_{t\in [0,T]}\|\mathbf{u}^\varepsilon\|_H^2\right)+\mathrm{E}\int_{0}^{T}\|\nabla \mathbf{u}^\varepsilon\|_{H}^2+\|\mathbf{u}^\varepsilon\|_{L^4}^4dt\leq C,
\end{align}
and for any $p\geq 2$
\begin{align}\label{3.2}
\mathrm{E}\left(\sup_{t\in [0,T]}\|\mathbf{u}^\varepsilon\|_H^p\right)+\mathrm{E}\int_{0}^{T}\|\mathbf{u}^\varepsilon\|_H^{p-2}\|\nabla \mathbf{u}^\varepsilon\|_{H}^2+\|\mathbf{u}^\varepsilon\|_H^{p-2}\|\mathbf{u}^\varepsilon\|_{L^4}^4dt\leq C,
\end{align}
where the positive constant $C$ is independent of  $\varepsilon$.
\end{lemma}
\begin{proof}
Using the It\^{o} formula to $\|\mathbf{u}^\varepsilon\|_{H}^2$, by \eqref{2.1*} we obtain
\begin{align}\label{a2.14}
\frac{1}{2}d\|\mathbf{u}^\varepsilon\|_H^2+(A^\varepsilon \mathbf{u}^\varepsilon, \mathbf{u}^\varepsilon)_{V'\times V}dt=(F(\mathbf{u}^\varepsilon, \mathfrak{L}_{\mathbf{u}^\varepsilon(t)}), \mathbf{u}^\varepsilon)dt+ (G(t, \mathbf{u}^\varepsilon)dW, \mathbf{u}^\varepsilon)+\frac{1}{2}\|G(t, \mathbf{u}^\varepsilon)\|_{L_2(H,H)}^2 dt.
\end{align}
The uniformly elliptic condition \eqref{1.2} leads to
\begin{align}\label{1.6}
(A^\varepsilon \mathbf{u}^\varepsilon, \mathbf{u}^\varepsilon)_{V'\times V}\geq \kappa\|\nabla \mathbf{u}^\varepsilon\|_{H}^2.
\end{align}
Using assumptions $(\mathbf{F.1})$ and $(\mathbf{G.1})$, we see
\begin{align}
(F(\mathbf{u}^\varepsilon, \mathfrak{L}_{\mathbf{u}^\varepsilon(t)}), \mathbf{u}^\varepsilon)\leq C(\|\mathbf{u}^\varepsilon\|_H^2+\mathrm{E}\|\mathbf{u}^\varepsilon\|_H^2-\|\mathbf{u}^\varepsilon\|_{L^4}^4),
\end{align}
and
\begin{align}\label{1.8}
\frac{1}{2}\|G(t, \mathbf{u}^\varepsilon)\|_{L_2(H,H)}^2 \leq C(1+\|\mathbf{u}^\varepsilon\|_H^{2}).
\end{align}
Using the Burkholder-Davis-Gundy inequality, the Young inequality and \eqref{1.8}, we get for any $p\geq 1$
\begin{align}\label{1.9}
&\mathrm{E}\left(\sup_{t\in [0,T]}\left|\int_{0}^{t}(G(r, \mathbf{u}^\varepsilon)dW, \mathbf{u}^\varepsilon)\right|^p\right)\nonumber\\
&\leq C\mathrm{E}\left|\int_{0}^{T}\sum_{k\geq 1}(G(t, \mathbf{u}^\varepsilon)\mathbf{e}_k, \mathbf{u}^\varepsilon)^2dt\right|^\frac{p}{2}\nonumber\\
&\leq C\mathrm{E}\left|\int_{0}^{T}\sum_{k\geq 1}\|G(t, \mathbf{u}^\varepsilon)\mathbf{e}_k\|_{H}^2 \|\mathbf{u}^\varepsilon\|_{H}^2dt\right|^\frac{p}{2}\nonumber\\
&\leq \frac{1}{4}\mathrm{E}\left(\sup_{t\in [0,T]}\|\mathbf{u}^\varepsilon\|_H^{2p}\right)+C\mathrm{E}\left(\int_{0}^{T}\| G(t, \mathbf{u}^\varepsilon)\|^2_{L_2(H,H)}dt\right)^p\nonumber\\
&\leq \frac{1}{4}\mathrm{E}\left(\sup_{t\in [0,T]}\|\mathbf{u}^\varepsilon\|_H^{2p}\right)+C\mathrm{E}\left(\int_{0}^{T}1+\|\mathbf{u}^\varepsilon\|_H^{2}dt\right)^p.
\end{align}

Combining \eqref{a2.14}-\eqref{1.8}, we have
\begin{align}\label{1.10}
\frac{1}{2}d\|\mathbf{u}^\varepsilon\|_H^2+\kappa\|\nabla \mathbf{u}^\varepsilon\|_{H}^2dt+C\|\mathbf{u}^\varepsilon\|_{L^4}^4dt\leq  C(1+\|\mathbf{u}^\varepsilon\|_H^{2}+\mathrm{E}\|\mathbf{u}^\varepsilon\|_H^2)dt+(G(t, \mathbf{u}^\varepsilon)dW, \mathbf{u}^\varepsilon).
\end{align}
Integrating of time, then taking supremum of $t$ over $[0,T]$ and expectation in \eqref{1.10}, by \eqref{1.9} with $p=1$ we see
\begin{align*}
&\mathrm{E}\left(\sup_{t\in [0,T]}\|\mathbf{u}^\varepsilon\|_H^2\right)+C\mathrm{E}\int_{0}^{T}\|\nabla \mathbf{u}^\varepsilon\|_{H}^2+\|\mathbf{u}^\varepsilon\|_{L^4}^4dt\nonumber\\
&\leq \|\mathbf{u}_0\|_H^2+ C\mathrm{E}\int_{0}^{T}(1+\|\mathbf{u}^\varepsilon\|_H^{2}+\mathrm{E}\|\mathbf{u}^\varepsilon\|_H^2)dt
+C\mathrm{E}\int_{0}^{T}1+\|\mathbf{u}^\varepsilon\|_H^{2}dt.
\end{align*}
The Gronwall lemma yields estimate \eqref{3.1}. Furthermore, for any $p\geq 2$, using the It\^{o} formula to $(\|\mathbf{u}^\varepsilon\|_{H}^2)^\frac{p}{2}$, by \eqref{a2.14} we have
\begin{align}\label{a1}
&\frac{p}{2}d\|\mathbf{u}^\varepsilon\|_H^p+p\|\mathbf{u}^\varepsilon\|_H^{p-2}(A^\varepsilon \mathbf{u}^\varepsilon, \mathbf{u}^\varepsilon)_{V'\times V}dt\nonumber\\
&=p\|\mathbf{u}^\varepsilon\|_H^{p-2}(F(\mathbf{u}^\varepsilon, \mathfrak{L}_{\mathbf{u}^\varepsilon(t)}), \mathbf{u}^\varepsilon)dt+ p\|\mathbf{u}^\varepsilon\|_H^{p-2}(G(t, \mathbf{u}^\varepsilon)dW, \mathbf{u}^\varepsilon)\nonumber\\
&\quad+\frac{p}{2}\|\mathbf{u}^\varepsilon\|_H^{p-2}\|G(t, \mathbf{u}^\varepsilon)\|_{L_2(H,H)}^2 dt+\frac{p(p-2)}{4}\|\mathbf{u}^\varepsilon\|_H^{p-4}(G(t, \mathbf{u}^\varepsilon), \mathbf{u}^\varepsilon)^2dt.
\end{align}
Similar to \eqref{1.6}-\eqref{1.8}, we obtain
\begin{align}\label{a1*}
&\frac{p}{2}d\|\mathbf{u}^\varepsilon\|_H^p+\kappa p\|\nabla \mathbf{u}^\varepsilon\|_{H}^2\|\mathbf{u}^\varepsilon\|_H^{p-2} dt+C\|\mathbf{u}^\varepsilon\|_H^{p-2}\|\mathbf{u}^\varepsilon\|_{L^4}^4dt\nonumber\\&\leq Cp\|\mathbf{u}^\varepsilon\|_H^{p-2}(1+\|\mathbf{u}^\varepsilon\|_H^2+\mathrm{E}\|\mathbf{u}^\varepsilon\|_H^2)dt+ p\|\mathbf{u}^\varepsilon\|_H^{p-2}(G(t, \mathbf{u}^\varepsilon)dW, \mathbf{u}^\varepsilon)\nonumber\\
&\quad+C(p)(1+\|\mathbf{u}^\varepsilon\|_H^{p})dt.
\end{align}
As \eqref{1.9}, we also have
\begin{align}\label{1.9**}
&\mathrm{E}\left(\sup_{t\in [0,T]}\left|\int_{0}^{t}p\|\mathbf{u}^\varepsilon\|_H^{p-2}(G(r, \mathbf{u}^\varepsilon)dW, \mathbf{u}^\varepsilon)\right|\right)\nonumber\\
& \leq \frac{p}{4}\mathrm{E}\left(\sup_{t\in [0,T]}\|\mathbf{u}^\varepsilon\|_H^{p}\right)+C(p)\mathrm{E}\int_{0}^{T}1+\|\mathbf{u}^\varepsilon\|_H^{p}dt.
\end{align}
Integrating of time, then taking supremum of $t$ over $[0,T]$ and expectation in \eqref{a1}, by \eqref{a1*} and \eqref{1.9**} we obtain
\begin{align}
&\mathrm{E}\left(\sup_{t\in [0,T]}\|\mathbf{u}^\varepsilon\|_H^p\right)+C\mathrm{E}\int_{0}^{T}\|\mathbf{u}^\varepsilon\|_H^{p-2}\|\nabla \mathbf{u}^\varepsilon\|_{H}^2+\|\mathbf{u}^\varepsilon\|_H^{p-2}\|\mathbf{u}^\varepsilon\|_{L^4}^4dt\nonumber\\
&\leq  \|\mathbf{u}_0\|_H^p+C(p)\mathrm{E}\int_{0}^{T}\|\mathbf{u}^\varepsilon\|_H^{p-2}(1+\|\mathbf{u}^\varepsilon\|_H^{2}+\mathrm{E}\|\mathbf{u}^\varepsilon\|_H^2)dt
+C(p)\mathrm{E}\int_{0}^{T}1+\|\mathbf{u}^\varepsilon\|_H^{p}dt\nonumber\\
&\leq \|\mathbf{u}_0\|_H^p+C(p)\mathrm{E}\int_{0}^{T}1+\|\mathbf{u}^\varepsilon\|_H^{p}dt,\nonumber
\end{align}
where we used the Young inequality in second inequality. Again, the Gronwall lemma could be used to obtain \eqref{3.2}.
\end{proof}

We next concern the time regularity of $\mathbf{u}^\varepsilon$ for the further compactness argument.
\begin{lemma}\label{lem3.2}Let $\mathbf{u}^\varepsilon$ be the solutions of system \eqref{Equ1.1}-\eqref{i1},  then there exists constant $C>0$ which is independent of $\varepsilon$ such that the time increment estimate satisfies
\begin{align*}
\mathrm{E}\left(\|\mathbf{u}^\varepsilon(t)-\mathbf{u}^\varepsilon(s)\|_{H^{-1, \frac{N}{N-1}}}^2\right)\leq C|t-s|,
\end{align*}
for any $t,s\in [0,T]$.
\end{lemma}
\begin{proof} Note that, $\mathbf{u}^\varepsilon(t)-\mathbf{u}^\varepsilon(s)$ satisfies
\begin{align}\label{3.15}
\mathbf{u}^\varepsilon(t)-\mathbf{u}^\varepsilon(s)=-\int_{s}^{t}A^\varepsilon\mathbf{u}^\varepsilon dr-\int_{s}^{t}B(\mathbf{u}^\varepsilon, \mathbf{u}^\varepsilon)dr+\int_{s}^{t}F(\mathbf{u}^\varepsilon, \mathfrak{L}_{\mathbf{u}^\varepsilon(r)})dr+\int_{s}^{t}G(r, \mathbf{u}^\varepsilon)dW.
\end{align}
Taking the $H^{-1,\frac{N}{N-1}}$-norm on both sides of \eqref{3.15} and then expectation, by the H\"{o}lder inequality and $H^{-1}\hookrightarrow H^{-1,\frac{N}{N-1}}$, we arrive at
\begin{align}\label{1.18}
&\mathrm{E}\left(\|\mathbf{u}(t)-\mathbf{u}(s)\|_{H^{-1,\frac{N}{N-1}}}^2\right)\nonumber\\
&\leq C\mathrm{E}\left(\int_{s}^{t}\|A^\varepsilon\mathbf{u}^\varepsilon\|_{H^{-1}} dr+\int_{s}^{t}\|B(\mathbf{u}^\varepsilon, \mathbf{u}^\varepsilon)\|_{H^{-1,\frac{N}{N-1}}}dr+\int_{s}^{t}\|F(\mathbf{u}^\varepsilon, \mathfrak{L}_{\mathbf{u}^\varepsilon(r)})\|_{H^{-1}}dr\right)^2\nonumber\\
&\leq C|t-s|\left(\mathrm{E}\int_{s}^{t}\|A^\varepsilon\mathbf{u}^\varepsilon\|_{H^{-1}}^2 dr+\mathrm{E}\int_{s}^{t}\|B(\mathbf{u}^\varepsilon, \mathbf{u}^\varepsilon)\|_{H^{-1, \frac{N}{N-1}}}^2 dr+\mathrm{E}\int_{s}^{t}\|F(\mathbf{u}^\varepsilon, \mathfrak{L}_{\mathbf{u}^\varepsilon(r)})\|_{H^{-1}}^2dr\right).
\end{align}
Note that using Lemma \ref{lem3.1}, we could verify that the last term is a square integrable martingale, hence this term vanishes after taking expectation. Using assumption of coefficient $a^\varepsilon_{i,j}\in L^\infty$, we can control the linear term
\begin{align}
\mathrm{E}\int_{s}^{t}\|A^\varepsilon\mathbf{u}^\varepsilon\|_{H^{-1}}^2 dr&\leq \mathrm{E}\int_{s}^{t}\left\|\sum_{i,j=1}^N\left(a^\varepsilon_{i,j}\frac{\partial \mathbf{u}^\varepsilon}{\partial x_j}\right)\right\|_{H}^2 dr\nonumber\\
&\leq \left\|\sum_{i,j=1}^Na^\varepsilon_{i,j}\right\|^2_{L^\infty}\mathrm{E}\int_{s}^{t}\left\|\sum_{j=1}^N\frac{\partial \mathbf{u}^\varepsilon}{\partial x_j}\right\|_{H}^2 dr\nonumber\\
&\leq C\mathrm{E}\int_{0}^{T}\|\nabla \mathbf{u}^\varepsilon\|_{H}^2dt.
\end{align}
From assumption \eqref{2.3},  we obtain
\begin{align}
\mathrm{E}\int_{s}^{t}\|B(\mathbf{u}^\varepsilon, \mathbf{u}^\varepsilon)\|_{H^{-1, \frac{N}{N-1}}}^2 dr\leq C\mathrm{E}\int_{0}^{T}\|\mathbf{u}^\varepsilon\|_{H}^2\|\mathbf{u}^\varepsilon\|_{V}^2dt.
\end{align}
Applying assumption ({\bf F.1}) again, we find
\begin{align}\label{1.21}
\mathrm{E}\int_{s}^{t}\|F(\mathbf{u}^\varepsilon, \mathfrak{L}_{\mathbf{u}^\varepsilon(r)})\|_{H^{-1}}^2dr&=\mathrm{E}\int_{s}^{t}\sup_{\varphi\in V, \|\varphi\|_{V}=1}(F(\mathbf{u}^\varepsilon, \mathfrak{L}_{\mathbf{u}^\varepsilon(r)}), \varphi)^2dr\nonumber\\
&\leq C\mathrm{E}\int_{s}^{t}\|\varphi\|_{V}^2(\|\mathbf{u}^\varepsilon\|_{H}^2+ \mathrm{E}\|\mathbf{u}^\varepsilon\|_{H}^2)dr\nonumber\\
&\leq C\mathrm{E}\int_{0}^{T}\|\mathbf{u}^\varepsilon\|_{H}^2dt.
\end{align}
The conclusion follows from \eqref{1.18}-\eqref{1.21} and Lemma \ref{lem3.1}.
\end{proof}

{\bf Tightness}. We consider the spaces
$$\mathscr{X}=C([0,T]; V')\cap L^2(0,T;H),$$
and
$$X=C([0,T]; V')\cap L^2(0,T;H)\cap L^2_w(0,T; V)\cap L_w^4(D\times (0,T)).$$
Denote by $\mathfrak{L}_{\mathbf{u}^\varepsilon}$ the law of $\mathbf{u}^\varepsilon$, we next show that the family of measures $\mathfrak{L}_{\mathbf{u}^\varepsilon}$ is tight in $X$, also in $\mathscr{X}$.

Indeed, for any $R>0$, define the set
$$\mathscr{B}_R:=\left\{\mathbf{u}^\varepsilon: \sup_{t\in [0,T]}\|\mathbf{u}^\varepsilon\|_{H}+\int_{0}^{T}\|\mathbf{u}^\varepsilon\|_{V}^2+
\|\mathbf{u}^\varepsilon\|_{L^4}^4dt+\int_{0}^{T}\left\|\frac{d\mathbf{u}^\varepsilon}{dt}\right\|_{H^{-1, \frac{N}{N-1}}}^2dt\leq R\right\}.$$
According to the Aubin-Lions lemma (see appendix, Lemma \ref{Simon1}) and the Banach-Alaoglu theorem, the set $\mathscr{B}_R$ is relative compact in $X$. By Lemmas \ref{lem3.1}, \ref{lem3.2} and the Chebyschev inequality, we see
$$
\mathrm{P}(\mathbf{u}^\varepsilon\in \mathscr{B}_R^c)$$
$$
\leq \frac{1}{R}\mathrm{E}\left(\sup_{t\in [0,T]}\|\mathbf{u}^\varepsilon\|_{H}+\int_{0}^{T}\|\mathbf{u}^\varepsilon\|_{V}^2+\|\mathbf{u}^\varepsilon\|_{L^4}^4dt
+\int_{0}^{T}\left\|\frac{d\mathbf{u}^\varepsilon}{d t}\right\|_{H^{-1,\frac{N}{N-1}}}^2dt\right)\leq \frac{C}{R},
$$
where $C>0$ is independent of $R, \varepsilon$, which implies the tightness of the family of measures $\mathfrak{L}_{\mathbf{u}^\varepsilon}$. Furthermore, we have the law $\mathfrak{L}_{\mathbf{u}^\varepsilon}\times \mathfrak{L}_W$ is tight in $X\times C([0,T]; H_0)$.

The Skorokhod-Jakubowski representative theorem will be used to represent a weakly convergent
probability measure sequence on a topology space as the distribution of a pointwise convergent
random variable sequence.

\begin{proposition}\label{pro3.1}\cite{brz} If $E$ is a topology space, and  there exists a sequence
of continuous functions $f_n: E\rightarrow R$ that separates points of $E$, denote by $\mathcal{B}$ the $\sigma$-algebra generated by $f_n$, then, it holds:

i.  every compact subset of $E$ is metrizable;

ii. if the set of probability measures $\{\mu_n\}_{n\geq 1}$ on $(E, \mathcal{B})$ is tight, then there exist a probability space $(\Omega, \mathcal{F}, \mathrm{P})$ and a sequence of random variables $\mathbf{u}_n, \mathbf{u}$ such that their laws are $\mu_n$, $\mu$ and $\mathbf{u}_n\rightarrow \mathbf{u}$, $\mathrm{P}$ a.s. as $n\rightarrow \infty$ in $E$.
\end{proposition}

Note that  since $C([0,T]; V')\cap L^2(0,T;H)$ is a Polish space,  there exists a countable set of continuous real-valued
functions separating points. For the space $L^2_w(0,T; V)$, since
$ L ^2(0,T;V)$
is a separable Hilbert space,
any countable dense subset
of $ \left ( L ^2(0,T;V)
\right )^*$
forms a sequence
of continuous
functions
on $ L _w^2(0,T;V))$
that separate points.
Similarly, we could show that the space $L_w^4(D\times (0,T))$ also satisfies the condition. With all conditions verified of Proposition \ref{pro3.1}, we could find that there exist a new probability space $(\widetilde{\Omega}, \widetilde{\mathcal{F}}, \widetilde{\mathrm{P}})$ and a sequence of random variables $\widetilde{\mathbf{u}}^\varepsilon, \widetilde{W}^\varepsilon, \widetilde{\mathbf{u}}, \widetilde{W}$ such that their laws are $\mathfrak{L}_{\mathbf{u}^\varepsilon}$, $\mathfrak{L}_W$, $\mathfrak{L}_{\mathbf{u}}$ and $$\widetilde{\mathbf{u}}^\varepsilon\rightarrow \widetilde{\mathbf{u}},~ {\rm in}~ X,$$
and
$$\mathfrak{L}_{\mathbf{u}^\varepsilon}\rightarrow \mathfrak{L}_{\mathbf{u}},~{\rm in}~\mathscr{P}(\mathscr{X}),$$
and
$$\widetilde{W}^\varepsilon\rightarrow \widetilde{W}, ~{\rm in} ~C([0,T]; H_0),$$
 $\widetilde{\mathrm{P}}$ a.s. as $\varepsilon\rightarrow 0$.

 Since the laws of $\widetilde{\mathbf{u}}^\varepsilon$ and $\mathbf{u}^\varepsilon$ coincide, they share the same estimates and the equations, thus we have
 \begin{align}\label{1.23**}
 \mathrm{E}^{\widetilde{\mathrm{P}}}\left(\sup_{t\in [0,T]}\|\widetilde{\mathbf{u}}^\varepsilon\|_H^p\right)+\mathrm{E}^{\widetilde{\mathrm{P}}}\int_{0}^{T}\|\widetilde{\mathbf{u}}^\varepsilon\|_H^{p-2}\|\nabla \widetilde{ \mathbf{u}}^\varepsilon\|_{H}^2+\|\widetilde{\mathbf{u}}^\varepsilon\|_H^{p-2}\|\widetilde{\mathbf{u}}^\varepsilon\|_{L^4}^4dt\leq C,
 \end{align}
for any $p\geq 2$, $ \mathrm{E}^{\widetilde{\mathrm{P}}}$  is the expectation with respect to $\widetilde{\mathrm{P}}$, the constant $C$ is independent of $\varepsilon$, and on the new probability space, $\widetilde{\mathrm{P}}$ a.s.
 \begin{equation*}
 \partial_t\widetilde{ \mathbf{u}}^\varepsilon+A^\varepsilon \widetilde{\mathbf{u}}^\varepsilon+B(\widetilde{\mathbf{u}}^\varepsilon, \widetilde{\mathbf{u}}^\varepsilon)=F(\widetilde{\mathbf{u}}^\varepsilon, \mathfrak{L}_{\widetilde{\mathbf{u}}^\varepsilon(t)})+G(t, \widetilde{\mathbf{u}}^\varepsilon)\frac{d\widetilde{W}^\varepsilon}{dt},
 \end{equation*}
in the weak sense, and with the initial law $\mathbb{P}=\delta_{\mathbf{u}_0^{-1}}$.

\section{Homogenization problem}
In this section, we prove the homogenization result. We begin with introducing some basic notations. Denote by $D_t=D\times [0,T]$, and  by $D_{\tau}=D^N\times \widetilde{T}=\left(-\frac{1}{2},\frac{1}{2}\right)^N\times \left(-\frac{1}{2}, \frac{1}{2}\right)$ which is the subset of $\mathbb{R}_y^N\times \mathbb{R}_\tau$. Denote by $L^p_{per}(D_{\tau})$ all the  $D_{\tau}$-periodic functions in $L_{loc}^p(\mathbb{R}^N\times \mathbb{R}_\tau)$, which is a Banach space after endowed with the norm
$$\|\mathbf{u}\|^p_{L^p_{per}(D_{\tau})}=\int_{D_\tau}|\mathbf{u}(y,\tau)|^pdyd\tau.$$

Denote by $C^\infty_{per}(D_\tau)$ all the $D_{\tau}$-periodic infinite differential functions on $\mathbb{R}^N\times \mathbb{R}_\tau$. We further define the space $V_{per}$ be the $D^N$-periodic functions in $V(\mathbb{R}_y^N)$ with the norm
$$\|\mathbf{u}\|^2_{V_{per}}=\int_{D^N}|\nabla\mathbf{u}(y)|^2dy,$$
which is a Hilbert space. Also define by the space $L^p(\widetilde{T}; V_{per})$ all measurable functions $\mathbf{u}: \widetilde{T}\rightarrow V_{per}$ which $\|\mathbf{u}(\tau)\|_{V_{per}}$ is integrable in $L^p(\widetilde{T})$, we endow it with the norm
$$\|\mathbf{u}\|^p_{L^p(\widetilde{T}; V_{per})}=\int_{\widetilde{T}}\|\mathbf{u}(\tau)\|_{V_{per}}^p d\tau,$$
which is a Banach space.

Now, we recall the concepts of $two$-$scale$ convergence.
\begin{definition} A sequence of $L^p(D_t)$-valued random variables $\mathbf{u}^\varepsilon$ is said to be $\Sigma$-weak convergence in $L^p(\Omega\times D_t)$ if there exists a certain $L^p(D_t, L^p_{per}(D_{\tau}))$-valued random variable $\mathbf{u}$ such that as $\varepsilon\rightarrow 0$,
$$\mathrm{E}\int_{D_t}\mathbf{u}^\varepsilon(x,t, \omega)\psi\left(x,t,\frac{x}{\varepsilon},\frac{t}{\varepsilon},\omega \right)dxdt\rightarrow \mathrm{E}\int_{D_t}\int_{D_\tau}\mathbf{u}(x,t,y,\tau,\omega)\psi(x,t,y,\tau,\omega)dxdydtd\tau,$$
for any $\psi\in L^{p'}(\Omega\times D_t, L^{p'}_{per}(D_{\tau}))$.
\end{definition}

\begin{lemma}\label{lem4.1} Suppose that $\mathbf{u}^\varepsilon$ is a sequence of $L^2(0,T;V)$-valued random variables with the regularity
$$\mathrm{E}\int_{0}^{T}\|\mathbf{u}^\varepsilon\|_{V}^2dt\leq C.$$
Then, there exist a $L^2(0,T;V)$-valued random variable $\mathbf{u}$ and a $L^2(D_t; L^2_{per}(D_{\tau}))$-valued random variable $\overline{\mathbf{u}}$ such that
\begin{align*}
\mathbf{u}^\varepsilon \rightarrow \mathbf{u},~ in~ L^2(\Omega; L^2(D_t)),
\end{align*}
and
\begin{align}\label{4.2*}
\frac{\partial \mathbf{u}^\varepsilon}{\partial x_i}\rightarrow \frac{\partial \mathbf{u}}{\partial x_i}+\frac{\partial \overline{\mathbf{u}}}{\partial y_i}, ~in ~ L^2(\Omega, L^2(D_t)), \Sigma\!-\!weak.
\end{align}
\end{lemma}

The following  version of compactness result should be more appropriate in our setting.
\begin{lemma}\cite[Theorem 4]{raz}\label{lem4.2} Under the assumptions of Lemma \ref{lem4.1}, and the assumption that
$$\mathbf{u}^\varepsilon\rightarrow \mathbf{u}, ~in ~L^2(D_t),~ \mathrm{P}~~ a.s.$$
then, there exist a subsequence (still denoted by $\mathbf{u}^\varepsilon$) and a $L^2(D_t; L^2_{per}(D_{\tau}))$-valued random variable $\overline{\mathbf{u}}$ such that \eqref{4.2*} holds.
\end{lemma}

In order to pass the limit of the nonlinear term, we also need the $\Sigma$-strong convergence.
\begin{definition}\label{def4.2} A sequence of $L^p(D_t)$-valued random variables $\mathbf{u}^\varepsilon$ is said to be $\Sigma$-strong convergence in $L^p(\Omega\times D_t)$ if there exists a certain $L^p(D_t, L^p_{per}(D_{\tau}))$-valued random variable $\mathbf{u}$ such that as $\varepsilon\rightarrow 0$,
$$\mathrm{E}\int_{D_t}\mathbf{u}^\varepsilon(x,t, \omega)\mathbf{v}^\varepsilon(x,t, \omega) dxdt\rightarrow \mathrm{E}\int_{D_t}\int_{D_\tau}\mathbf{u}(x,t,y,\tau,\omega)\mathbf{v}(x,t,y,\tau,\omega)dxdydtd\tau,$$
for any bounded $\mathbf{v}^\varepsilon\in L^{p'}(\Omega\times D_t)$ with $\mathbf{v}^\varepsilon \rightarrow \mathbf{v}$ in $L^{p'}(\Omega\times D_t)$, $\Sigma$-weak.
\end{definition}

\begin{lemma}\label{lem4} \cite{zhi} A sequence of $L^p(D_t)$-valued random variables $\mathbf{u}^\varepsilon$ is said to be $\Sigma$-strong convergence in $L^p(\Omega\times D_t)$ if  there exists a certain $L^p(D_t, L^p_{per}(D_{\tau}))$-valued random variable $\mathbf{u}$ such that

i. the $\Sigma$-weak convergence holds;

ii. it satisfies
$$\|\mathbf{u}^\varepsilon\|_{L^p(\Omega\times D_t)}\rightarrow \|\mathbf{u}\|_{L^p(\Omega\times D_t, L^p_{per}(D_{\tau}))}.$$
\end{lemma}

The following lemma provides a way to pass the limit of a sequence of product functions.
\begin{lemma}\label{lem4.3}\cite{zhi} Assume that for any $r, p,q>1$ with $\frac{1}{r}=\frac{1}{p}+\frac{1}{q}$, if the following two conditions hold

i. a sequence of $L^p(D_t)$-valued random variables $\mathbf{u}^\varepsilon$ is $\Sigma$-weak convergence to some certain $\mathbf{u}\in L^p(\Omega\times D_t, L^p_{per}(D_{\tau}))$;

ii. a sequence of $L^q(D_t)$-valued random variables $\mathbf{v}^\varepsilon$ is $\Sigma$-strong convergence to some certain $\mathbf{v}\in L^q(\Omega\times D_t, L^q_{per}(D_{\tau}))$.

Then, we have the sequence of $\mathbf{u}^\varepsilon\mathbf{v}^\varepsilon$ is $\Sigma$-weak convergence to $\mathbf{u}\mathbf{v}$ in $L^r(\Omega\times D_t)$.
\end{lemma}

Let $\widetilde{\mathbf{u}}^{\varepsilon}$ be the sequence we chosen from the Skorokhod-Jakubowski representative theorem. To simplify the notations, we still use $\mathbf{u}^\varepsilon, \mathrm{E}$ instead of $\widetilde{\mathbf{u}}^{\varepsilon}, \mathrm{E}^{\widetilde{\mathrm{P}}}$. We already known that from Proposition \ref{pro3.1}, $\mathrm{P}$ a.s.
\begin{align}\label{2.1}
\mathbf{u}^\varepsilon\rightarrow \mathbf{u}~ {\rm in}~ X, ~~\mathfrak{L}_{\mathbf{u}^\varepsilon}\rightarrow \mathfrak{L}_{\mathbf{u}},~{\rm in}~\mathscr{P}(\mathscr{X}),
\end{align}
and by Lemma \ref{lem4.2}, there exists $\overline{\mathbf{u}}\in L^2(\Omega\times D_t; L^2_{per}(D_{\tau}))$  such that
\begin{align}\label{4.4}
\frac{\partial \mathbf{u}^\varepsilon}{\partial x_i}\rightarrow \frac{\partial \mathbf{u}}{\partial x_i}+\frac{\partial \overline{\mathbf{u}}}{\partial y_i}, ~{\rm in} ~ L^2(\Omega\times (0,T);H), \Sigma\!-\!{\rm weak},
\end{align}
as $\varepsilon\rightarrow 0$. Also, the sequence of $\mathbf{u}^\varepsilon$ satisfies the energy estimate
\begin{align}\label{1.23}
 \mathrm{E}\left(\sup_{t\in [0,T]}\|\mathbf{u}^\varepsilon\|_H^p\right)+\mathrm{E}\int_{0}^{T}\|\mathbf{u}^\varepsilon\|_H^{p-2}\|\nabla \mathbf{u}^\varepsilon\|_{H}^2+\|\mathbf{u}^\varepsilon\|_H^{p-2}\|\mathbf{u}^\varepsilon\|_{L^4}^4dt\leq C,
 \end{align}
for any $p\geq 2$, where the positive constant $C$ is independent of $\varepsilon$.

Define the spaces
$$\mathbb{X}=V\times L^2(D; L^2(\widetilde{T}; V_{per})),$$
for any $\mathbf{u}=(\mathbf{u}_0, \mathbf{u}_1)$, with the norm
$$\|\mathbf{u}\|_{\mathbb{X}}=\|\mathbf{u}_0\|_{V}+\|\mathbf{u}_1\|_{ L^2(D; L^2(\widetilde{T}; V_{per}))},$$
and
$$\widetilde{\mathbb{X}}=L^2(0,T; V)\times L^2(D_t; L^2(\widetilde{T}; V_{per})),$$
with the norm
$$\|\mathbf{u}\|_{\widetilde{\mathbb{X}}}=\|\mathbf{u}_0\|_{L^2(0,T; V)}+\|\mathbf{u}_1\|_{ L^2(D_t; L^2(\widetilde{T}; V_{per}))}.$$

{\bf The limit equations}. We have that the couple $(\mathbf{u}, \overline{\mathbf{u}})$ solves the following variational problem.

\begin{proposition}\label{pro2.1} Under the assumptions $({\bf F.2})$, $({\bf G.1})$, the couple $(\mathbf{u}, \overline{\mathbf{u}})$ satisfies the equations $\mathrm{P}$ a.s.
\begin{align}\label{4.3}
&\int_{0}^{T}(\mathbf{u}'(t), \varphi)dt\nonumber\\&=-\sum_{i,j=1}^N\int_{D_t}\int_{D_\tau}a_{i,j}\left(\frac{\partial \mathbf{u}}{\partial x_i}+\frac{\partial \overline{\mathbf{u}}}{\partial y_i}\right)\left(\frac{\partial \varphi}{\partial x_j}+\frac{\partial \psi}{\partial y_j}\right)dxdydtd\tau\nonumber\\
&\quad-\int_{0}^{T}(B(\mathbf{u}, \mathbf{u}), \varphi)dt+\int_{0}^{T}(F(\mathbf{u}, \mathfrak{L}_{\mathbf{u}(t)}), \varphi)dt+\int_{0}^{T}(G(t, \mathbf{u}), \varphi)dW,
\end{align}
for any $(\varphi,\psi)\in L^2(\Omega; \widetilde{\mathbb{X}})$, $T>0$.
\end{proposition}
\begin{proof} First, for any $t\in [0,T]$ denote
\begin{align*}
\mathscr{M}_{\Psi^\varepsilon}^\varepsilon(t)&=\int_{0}^{t}((\mathbf{u}^\varepsilon)', \Psi^\varepsilon)dr+\int_{0}^{t}(A^\varepsilon \mathbf{u}^\varepsilon, \Psi^\varepsilon)_{V'\times V}dr+\int_{0}^{t}(B(\mathbf{u}^\varepsilon, \mathbf{u}^\varepsilon), \Psi^\varepsilon)dr\\&\quad-\int_{0}^{t}(F(\mathbf{u}^\varepsilon, \mathfrak{L}_{\mathbf{u}^\varepsilon(r)}), \Psi^\varepsilon)dr,
\end{align*}
where $\Psi^\varepsilon=\left(\varphi(x,t)+\varepsilon\psi\left(x,t,\frac{x}{\varepsilon}, \frac{t}{\varepsilon}\right)\right)1_{\mathcal{A}}(\omega)$, $(x,t)\in D_t$, for any $\varphi\in C^\infty(D_t), \psi\in C^\infty(D_t)\times C^\infty_{per}( D_\tau) $, and $1_{\cdot}$ is the indicator function, $\mathcal{A}\in \mathcal{B}(\Omega)$. We have $\Psi^\varepsilon\in L^2(\Omega; C^\infty(D_t))$. Note that $$\mathscr{M}_{\Psi^\varepsilon}^\varepsilon(t)=\int_{0}^{t}(G(r, \mathbf{u}^\varepsilon), \Psi^\varepsilon)dW$$ is a square integrable martingale with quadratic variation
$$\langle\!\langle \mathscr{M}_{\Psi}^\varepsilon \rangle\!\rangle=\int_{0}^{t}\|G^*(r, \mathbf{u}^\varepsilon)\Psi^\varepsilon\|^2_{H}dr,$$
where $G^*$ is the adjoint operator of $G$.

And, denote
\begin{align*}
\mathscr{M}(t)=&\int_{0}^{t}(\mathbf{u}'(r), \varphi1_{\mathcal{A}}(\omega))dr+\sum_{i,j=1}^N\int_{D_r}\int_{D_\tau}a_{i,j}\left(\frac{\partial \mathbf{u}}{\partial x_i}+\frac{\partial \overline{\mathbf{u}}}{\partial y_i}\right)\left(\frac{\partial \varphi}{\partial x_j}+\frac{\partial \psi}{\partial y_j}\right)1_{\mathcal{A}}(\omega)dxdydrd\tau\nonumber\\
&+\int_{0}^{t}(B(\mathbf{u}, \mathbf{u}), \varphi1_{\mathcal{A}}(\omega))dr
-\int_{0}^{t}(F(\mathbf{u}, \mathfrak{L}_{\mathbf{u}(r)}), \varphi1_{\mathcal{A}}(\omega))dr,
\end{align*}
where $D_r=D\times [0,t]$.
Then, it holds
\begin{align}\label{2.6}
&\mathrm{E}(\mathscr{M}_{\Psi^\varepsilon}^\varepsilon (t)-\mathscr{M}(t))\nonumber\\
&=\mathrm{E}\int_{0}^{t}((\mathbf{u}^\varepsilon)', \Psi^\varepsilon)-(\mathbf{u}'(r), \varphi1_{\mathcal{A}}(\omega))dr\nonumber\\
&\quad+\mathrm{E}\left(\int_{0}^{t}(A^\varepsilon \mathbf{u}^\varepsilon, \Psi^\varepsilon)_{V'\times V}dr-\sum_{i,j=1}^N\int_{D_r}\int_{D_\tau}a_{i,j}\left(\frac{\partial \mathbf{u}}{\partial x_i}+\frac{\partial \overline{\mathbf{u}}}{\partial y_i}\right)\left(\frac{\partial \varphi}{\partial x_j}+\frac{\partial \psi}{\partial y_j}\right)1_{\mathcal{A}}(\omega)dxdydrd\tau\right)\nonumber\\
&\quad+\mathrm{E}\int_{0}^{t}(B(\mathbf{u}^\varepsilon, \mathbf{u}^\varepsilon), \Psi^\varepsilon)-(B(\mathbf{u}, \mathbf{u}), \varphi1_{\mathcal{A}}(\omega))dr\nonumber\\
&\quad-\mathrm{E}\int_{0}^{t}(F(\mathbf{u}^\varepsilon, \mathfrak{L}_{\mathbf{u}^\varepsilon(t)}), \Psi^\varepsilon)-(F(\mathbf{u}, \mathfrak{L}_{\mathbf{u}(r)}), \varphi1_{\mathcal{A}}(\omega))dr.
\end{align}
We next verify that the right-hand side terms of \eqref{2.6} converge to zero as $\varepsilon\rightarrow 0$.  A simple calculation gives
\begin{align}\label{2.7}
\mathrm{E}\int_{0}^{t}((\mathbf{u}^\varepsilon)', \Psi^\varepsilon)dr=\mathrm{E}\int_{0}^{t}((\mathbf{u}^\varepsilon)', \left(\varphi+\varepsilon\psi\right)1_{\mathcal{A}}(\omega))dr
=\mathrm{E}\int_{0}^{t}\left(\mathbf{u}^\varepsilon, \frac{\partial\varphi}{\partial r}+\varepsilon\frac{\partial\psi}{\partial r}+\frac{\partial\psi}{\partial \tau}\right)1_{\mathcal{A}}(\omega)dr.
\end{align}
The H\"{o}lder inequality and \eqref{1.23} yield
\begin{align*}
\mathrm{E}\int_{0}^{t}\left(\mathbf{u}^\varepsilon, \varepsilon\frac{\partial\psi}{\partial r}1_{\mathcal{A}}(\omega)\right)dr\leq \varepsilon \mathrm{E}\int_{0}^{t}\|\mathbf{u}^\varepsilon\|_{H}\left\|\frac{\partial\psi}{\partial r}\right\|_{H} dr\leq C\varepsilon,
\end{align*}
as well as \eqref{2.1} lead to
\begin{align}\label{2.8}
&\mathrm{E}\int_{0}^{t}\left(\mathbf{u}^\varepsilon, \frac{\partial\varphi}{\partial r}+\left(\varepsilon\frac{\partial\psi}{\partial r}+\frac{\partial\psi}{\partial \tau}\right)1_{\mathcal{A}}(\omega)\right)dr\nonumber\\
&\rightarrow \mathrm{E}\int_{0}^t\left(\mathbf{u}, \frac{\partial\varphi}{\partial r}1_{\mathcal{A}}(\omega)\right)dr+\mathrm{E}\int_{D_r}\int_{D_\tau}\mathbf{u}\frac{\partial\psi}{\partial \tau}1_{\mathcal{A}}(\omega)dxdydrd\tau.
\end{align}
Then, utilizing the periodicity of $\tau$, we see
\begin{align}\label{2.9}
\mathrm{E}\int_{D_r}\int_{D_\tau}\mathbf{u}\frac{\partial\psi}{\partial \tau}1_{\mathcal{A}}(\omega)dxdydrd\tau=\mathrm{E}\int_{D_r}1_{\mathcal{A}}(\omega)\mathbf{u}\left(\int_{D_\tau}\frac{\partial\psi}{\partial \tau}d\tau dy\right)drdx=0.
\end{align}
Combining \eqref{2.7}-\eqref{2.9}, we have
\begin{align}\label{2.10}
\mathrm{E}\int_{0}^{t}((\mathbf{u}^\varepsilon)', \Psi^\varepsilon)dr\rightarrow \mathrm{E}\int_{0}^t\left(\mathbf{u}, \frac{\partial\varphi}{\partial r}1_{\mathcal{A}}(\omega)\right)dr=\mathrm{E}\int_{0}^{t}(\mathbf{u}'(r), \varphi1_{\mathcal{A}}(\omega))dr.
\end{align}
Furthermore, for the second term on the right-hand side of \eqref{2.6}, using \eqref{4.4} and \eqref{1.23} we deduce
\begin{align}\label{4.13}
&\mathrm{E}\int_{0}^{t}(A^\varepsilon \mathbf{u}^\varepsilon, \Psi^\varepsilon)_{V'\times V}dr=\mathrm{E}\int_{0}^{t}\sum_{i,j=1}^N\left(a^\varepsilon_{i,j}\frac{\partial\mathbf{u}^\varepsilon}{\partial x_i}, \frac{\partial\Psi^\varepsilon}{\partial x_j}\right)dr\nonumber\\
&=\mathrm{E}\int_{0}^{t}\sum_{i,j=1}^N\left(a^\varepsilon_{i,j}\frac{\partial\mathbf{u}^\varepsilon}{\partial x_i}, \frac{\partial\varphi}{\partial x_j}+\frac{\partial\psi}{\partial y_j}+\varepsilon\frac{\partial\psi}{\partial x_j}\right)1_{\mathcal{A}}(\omega)dr\nonumber\\
&\rightarrow \mathrm{E}\int_{D_r}\int_{D_\tau}\sum_{i,j=1}^Na_{i,j}\left(\frac{\partial \mathbf{u}}{\partial x_i}+\frac{\partial \overline{\mathbf{u}}}{\partial y_i}\right)\left(\frac{\partial \varphi}{\partial x_j}+\frac{\partial \psi}{\partial y_j}\right)1_{\mathcal{A}}(\omega)dxdydrd\tau.
\end{align}

The next point is to pass the limit of nonlinear term $B(\mathbf{u}^\varepsilon, \mathbf{u}^\varepsilon)$.
Using \eqref{2.1}, \eqref{1.23} and the Vitali convergence theorem (see appendix, Theorem \ref{thm4.1}), we have
$$\mathbf{u}^\varepsilon\rightarrow \mathbf{u},~ {\rm in}~ L^2(\Omega; L^2(0,T;H)),$$
as well as the convergence
$$\frac{\partial \mathbf{u}^\varepsilon}{\partial x_i}\rightarrow \frac{\partial \mathbf{u}}{\partial x_i}+\frac{\partial \overline{\mathbf{u}}}{\partial y_i}, ~{\rm in} ~ L^2(\Omega, L^2(0,T;H)), \Sigma\!-\!{\rm weak}, $$
lead to (Lemma \ref{lem4.3})
\begin{align}\label{4.14}
\mathrm{E}\int_{0}^{t}(B(\mathbf{u}^\varepsilon, \mathbf{u}^\varepsilon),\varphi 1_{\mathcal{A}}(\omega))dr\rightarrow \mathrm{E}\int_{D_r}\int_{D_\tau}\sum_{i=1}^N\mathbf{u}_i\left(\frac{\partial \mathbf{u}}{\partial x_i}+\frac{\partial \overline{\mathbf{u}}}{\partial y_i}\right) \varphi1_{\mathcal{A}}(\omega) dxdydrd\tau.
\end{align}
By \eqref{2.3} and \eqref{1.23}, we obtain
\begin{align*}
\mathrm{E}\int_{0}^{t}(B(\mathbf{u}^\varepsilon, \mathbf{u}^\varepsilon), \varepsilon\psi1_{\mathcal{A}}(\omega))dr\nonumber &\leq \varepsilon\mathrm{E}\int_{0}^{t}\|\psi\|_{H^{1,N}}\|B(\mathbf{u}^\varepsilon, \mathbf{u}^\varepsilon)\|_{H^{-1, \frac{N}{N-1}}}dr\nonumber\\
&\leq \varepsilon\|\psi\|_{H^{1,N}}\mathrm{E}\int_{0}^{t}\|\mathbf{u}\|_{H}\|\mathbf{u}\|_{V}dr\leq C\varepsilon,
\end{align*}
giving rise to
\begin{align}\label{4.15}
\mathrm{E}\int_{0}^{t}(B(\mathbf{u}^\varepsilon, \mathbf{u}^\varepsilon),\varepsilon\psi 1_{\mathcal{A}}(\omega))dr\rightarrow 0.
\end{align}
From \eqref{4.14} and \eqref{4.15}, we get
\begin{align*}
\mathrm{E}\int_{0}^{t}(B(\mathbf{u}^\varepsilon, \mathbf{u}^\varepsilon),\left( \varphi+\varepsilon\psi\right)1_{\mathcal{A}}(\omega))dr\rightarrow \mathrm{E}\int_{D_r}\int_{D_\tau}\sum_{i=1}^N\mathbf{u}_i\left(\frac{\partial \mathbf{u}}{\partial x_i}+\frac{\partial \overline{\mathbf{u}}}{\partial y_i}\right) \varphi1_{\mathcal{A}}(\omega) dxdydrd\tau.
\end{align*}
Using the periodicity of $y$, we have
\begin{align*}
&\mathrm{E}\int_{D_r}\int_{D_\tau}\sum_{i=1}^N\mathbf{u}_i\left(\frac{\partial \mathbf{u}}{\partial x_i}+\frac{\partial \overline{\mathbf{u}}}{\partial y_i}\right) \varphi1_{\mathcal{A}}(\omega)  dxdydrd\tau \nonumber\\
&=\mathrm{E}\int_{0}^t(B(\mathbf{u}, \mathbf{u}), \varphi1_{\mathcal{A}}(\omega))dr+\mathrm{E}\int_{D_r}\varphi1_{\mathcal{A}}(\omega)\sum_{i=1}^N\mathbf{u}_i\left(\int_{D_\tau}\frac{\partial \overline{\mathbf{u}}}{\partial y_i}dyd\tau\right)dxdr\nonumber\\
&=\mathrm{E}\int_{0}^{t}(B(\mathbf{u}, \mathbf{u}), \varphi1_{\mathcal{A}}(\omega))dr.
\end{align*}
Hence we obtain
\begin{align}\label{4.16}
\mathrm{E}\int_{0}^{t}(B(\mathbf{u}^\varepsilon, \mathbf{u}^\varepsilon),\left( \varphi+\varepsilon\psi\right)1_{\mathcal{A}}(\omega))dr\rightarrow \mathrm{E}\int_{0}^{t}(B(\mathbf{u}, \mathbf{u}), \varphi1_{\mathcal{A}}(\omega))dr.
\end{align}

We next pass the limit of $F(\mathbf{u}^\varepsilon, \mathfrak{L}_{\mathbf{u}^\varepsilon(t)})$. Since
$$\mathfrak{L}_{\mathbf{u}^\varepsilon(t)}\rightarrow \mathfrak{L}_{\mathbf{u}(t)}, ~{\rm in} ~\mathscr{P}(\mathscr{X}),$$
as a result of \eqref{1.23} and \cite[Thereom 5.5]{CD}, we have
$$\mathfrak{L}_{\mathbf{u}^\varepsilon(t)}\rightarrow \mathfrak{L}_{\mathbf{u}(t)}, ~{\rm~in} ~\mathscr{P}_2(\mathscr{X}).$$
Following from the definition of $p$-Wasserstein distance, we see
$$\mathcal{W}_{2,T,V'}(\mathfrak{L}_{\mathbf{u}^\varepsilon(t)}, \mathfrak{L}_{\mathbf{u}(t)})\rightarrow 0,$$
as $\varepsilon\rightarrow 0$. Using the fact that
$$\mathcal{W}_{2,V'}(\mathfrak{L}_{\mathbf{u}^\varepsilon(t)}, \mathfrak{L}_{\mathbf{u}(t)})\leq \mathcal{W}_{2,T,V'}(\mathfrak{L}_{\mathbf{u}^\varepsilon(t)}, \mathfrak{L}_{\mathbf{u}(t)}),$$
we obtain
$$\mathfrak{L}_{\mathbf{u}^\varepsilon(t)}\rightarrow \mathfrak{L}_{\mathbf{u}(t)}, ~{\rm~in} ~\mathscr{P}_2(V').$$
Then, using the assumption $({\bf F.2})$, \eqref{2.1}, we know $\mathrm{P}$ a.s.
\begin{align*}
\int_{0}^{t}(F(\mathbf{u}^\varepsilon, \mathfrak{L}_{\mathbf{u}^\varepsilon(r)}), \Psi^\varepsilon)dr\rightarrow \int_{0}^{t}(F(\mathbf{u}, \mathfrak{L}_{\mathbf{u}(r)}), \varphi1_{\mathcal{A}}(\omega))dr.
\end{align*}
Finally, the dominated convergence theorem yields
\begin{align}\label{2.14}
\mathrm{E}\int_{0}^{t}(F(\mathbf{u}^\varepsilon, \mathfrak{L}_{\mathbf{u}^\varepsilon(r)}), \Psi^\varepsilon)dr\rightarrow \mathrm{E}\int_{0}^{t}(F(\mathbf{u}, \mathfrak{L}_{\mathbf{u}(r)}), \varphi1_{\mathcal{A}}(\omega))dr.
\end{align}

Combining \eqref{2.6}, \eqref{2.10}, \eqref{4.13}, \eqref{4.16} and \eqref{2.14}, we conclude
\begin{align}\label{2.15}
\mathrm{E}(\mathscr{M}_{\Psi^\varepsilon}^\varepsilon (t)-\mathscr{M}(t))\rightarrow 0, ~{\rm as} ~\varepsilon\rightarrow 0.
\end{align}
Let $\Phi$ be any bounded continuous function on $X\times C([0,T];H_0)$, by \eqref{2.15} and the martingale property,
\begin{align*}
\mathrm{E}\left((\mathscr{M}(t)-\mathscr{M}(s))\Phi((\mathbf{u}, W)|_{[0,s]})\right)=\lim_{\varepsilon\rightarrow 0}\mathrm{E}\left((\mathscr{M}_{\Psi^\varepsilon}^\varepsilon(t)-\mathscr{M}_{\Psi^\varepsilon}^\varepsilon(s))\Phi((\mathbf{u}^\varepsilon, W^\varepsilon)|_{[0,s]})\right)=0.
\end{align*}
From the arbitrariness of $\Phi$, we finally obtain
\begin{align}\label{2.17}
\mathrm{E}(\mathscr{M}(t)|\mathscr{F}_s)=\mathscr{M}(s),
\end{align}
where the  filtration $\{\mathscr{F}_t\}_{t\geq 0}$ is generated by $\sigma\{\mathbf{u}(s), W(s), s\leq t\}$ satisfying the usual conditions.

We proceed to show that
\begin{align}\label{2.18}
\mathrm{E}\left(\mathscr{M}^2(t)-\int_{0}^{t}\|G^*(r,\mathbf{u})\varphi1_{\mathcal{A}}(\omega)\|^2_{H}dr\bigg|\mathscr{F}_s\right)
=\mathscr{M}^2(s)-\int_{0}^{s}\|G^*(r, \mathbf{u})\varphi1_{\mathcal{A}}(\omega)\|^2_{H}dr.
\end{align}
As \eqref{1.9}, we have
\begin{align*}
\mathrm{E}|\mathscr{M}^\varepsilon_{\Phi^\varepsilon}(t)|^p\leq C\mathrm{E}\left|\int_{0}^{T}\sum_{k\geq 1}(G(t, \mathbf{u}^\varepsilon)\mathbf{e}_k, \mathbf{u}^\varepsilon)^2dt\right|^\frac{p}{2}\leq C,
\end{align*}
where the positive constant $C$ is independent of $\varepsilon$. Then, we infer from the Vitali convergence theorem  that
\begin{align}\label{4.24*}\mathrm{E}(\mathscr{M}^\varepsilon_{\Phi^\varepsilon}(t)-\mathscr{M}(t))^2\rightarrow 0, ~{\rm as} ~\varepsilon\rightarrow 0.\end{align}

We next show that
\begin{align}\label{4.26*}\mathrm{E}\left|\int_{0}^{t}\|G^*(r,\mathbf{u}^\varepsilon)\Phi^\varepsilon\|^2_{H}
-\|G^*(r,\mathbf{u})\varphi1_{\mathcal{A}}(\omega)\|^2_{H}dr\right|\rightarrow 0,\end{align}
as $\varepsilon\rightarrow 0$.
Since $G$ is a Hilbert-Schmidt operator in $L_2(H;H)$, we have the adjoint operator $G^*\in L_2(H;H)$ and $\|G\|_{L_2(H;H)}=\|G^*\|_{L_2(H;H)}$.
From assumption $({\bf G.2})$ and \eqref{1.23}, we have
\begin{align}\label{4.27*}\mathrm{E}\left|\int_{0}^{t}\|G^*(r,\mathbf{u}^\varepsilon)\varepsilon\psi1_{\mathcal{A}}(\omega)\|^2_{H}dr\right|&\leq \varepsilon\mathrm{E}\left|\int_{0}^{t}\|\psi\|_{C_{per}(D_\tau)}\|G^*(r,\mathbf{u}^\varepsilon)\|^2_{L_2(H;H)}dr\right|\nonumber\\
&\leq \varepsilon C\|\psi\|_{C_{per}(D_\tau)}\mathrm{E}\int_{0}^{t}(1+\|\mathbf{u}^\varepsilon\|_{H}^2)dr\rightarrow 0,\end{align}
as $\varepsilon\rightarrow 0$. It remains to show that
\begin{align}\label{4.22}\mathrm{E}\left|\int_{0}^{t}\|G^*(r,\mathbf{u}^\varepsilon)\varphi1_{\mathcal{A}}(\omega)\|^2_{H}
-\|G^*(r,\mathbf{u})\varphi1_{\mathcal{A}}(\omega)\|^2_{H}dr\right|\rightarrow 0.\end{align}
By \eqref{2.1} and assumption $({\bf G.1})$, we have $\mathrm{P}$ a.s.
\begin{align*}&\int_{0}^{t}\|(G^*(r,\mathbf{u}^\varepsilon)-G^*(r,\mathbf{u}))\varphi1_{\mathcal{A}}(\omega)\|^2_{H}dr\nonumber\\
&\leq \int_{0}^{t}\|\varphi\|_{C(D_t)}\|G^*(r,\mathbf{u}^\varepsilon)-G^*(r,\mathbf{u})\|^2_{L_2(H;H)}dr\nonumber\\
&\leq C\|\varphi\|_{C(D_t)}\int_{0}^{t}\|\mathbf{u}^\varepsilon-\mathbf{u}\|^2_{H}dr\rightarrow 0.\end{align*}
Again, the dominated convergence theorem yields
\begin{align*}&\mathrm{E}\left|\int_{0}^{t}\|G^*(r,\mathbf{u}^\varepsilon)\varphi1_{\mathcal{A}}(\omega)\|^2_{H}
-\|G^*(r,\mathbf{u})\varphi1_{\mathcal{A}}(\omega)\|^2_{H}dr\right|\nonumber\\
&\leq \mathrm{E}\int_{0}^{t}\|(G^*(r,\mathbf{u}^\varepsilon)-G^*(r,\mathbf{u}))\varphi1_{\mathcal{A}}(\omega)\|^2_{H}dr
\rightarrow 0,\end{align*}
as desired. The limit \eqref{4.26*} is a consequence of \eqref{4.27*} and \eqref{4.22}.

Using \eqref{4.24*} and \eqref{4.26*}, we further obtain
\begin{align}\label{2.18*}
&\mathrm{E}\left(\!\left(\left(\mathscr{M}^\varepsilon_{\Phi^\varepsilon}(t)\right)^2-\left(\mathscr{M}^\varepsilon_{\Phi^\varepsilon}(s)\right)^2
-\int_{0}^{t}\|G^*(r,\mathbf{u}^\varepsilon)\Phi^\varepsilon\|^2_{H}dr
+\int_{0}^{s}\|G^*(r,\mathbf{u}^\varepsilon)\Phi^\varepsilon\|^2_{H}dr\right)\Phi((\mathbf{u}^\varepsilon, W^\varepsilon)|_{[0,s]})\right)\nonumber\\
&\rightarrow\mathrm{E}\left(\!\left(\mathscr{M}^2(t)-\mathscr{M}^2(s)\!-\!\int_{0}^{t}\|G^*(r,\mathbf{u})\varphi1_{\mathcal{A}}(\omega)\|^2_{H}dr
\!+\!\int_{0}^{s}\|G^*(r,\mathbf{u})\varphi1_{\mathcal{A}}(\omega)\|^2_{H}dr\right)\!\Phi((\mathbf{u}, W)|_{[0,s]})\!\right).
\end{align}
By the martingale property, we deduce
$$\mathrm{E}\left(\!\left(\left(\mathscr{M}^\varepsilon_{\Phi^\varepsilon}(t)\right)^2\!-\!\left(\mathscr{M}^\varepsilon_{\Phi^\varepsilon}(s)\right)^2
\!-\!\int_{0}^{t}\|G^*(r,\mathbf{u}^\varepsilon)\Phi^\varepsilon\|^2_{H}dr
\!+\!\int_{0}^{s}\|G^*(r,\mathbf{u}^\varepsilon)\Phi^\varepsilon\|^2_{H}dr\right)\!\Phi((\mathbf{u}^\varepsilon, W^\varepsilon)|_{[0,s]})\right)=0,$$
then, by \eqref{2.18*} we further obtain
$$\mathrm{E}\left(\!\left(\mathscr{M}^2(t)\!-\!\mathscr{M}^2(s)\!-\!\int_{0}^{t}\|G^*(r,\mathbf{u})\varphi1_{\mathcal{A}}(\omega)\|^2_{H}dr
\!+\!\int_{0}^{s}\|G^*(r,\mathbf{u})\varphi1_{\mathcal{A}}(\omega)\|^2_{H}dr\right)\!\Phi((\mathbf{u}, W)|_{[0,s]})\right)=0.$$
The arbitrariness of $\Phi$ yields \eqref{2.18}. By \eqref{2.17}, \eqref{2.18} and \eqref{1.23}, the martingale representative theorem implies that
\begin{align*}
\mathscr{M}(t)=\int_{0}^{t}(G(r,\mathbf{u}), \varphi1_{\mathcal{A}}(\omega))dW.
\end{align*}
By the density of $C^\infty(D_t)\times 1_{\cdot}(\omega),  C^\infty(D_t)\times C^\infty_{per}( D_\tau) \times 1_{\cdot}(\omega)$ in $L^2(\Omega; \widetilde{\mathbb{X}})$, we complete the proof of Proposition \ref{pro2.1}.
\end{proof}

We will use the following modified Yamada-Watanabe theorem given by \cite[Lemma 2.1]{HW} to establish the well-posedness of strong solution to equations \eqref{4.3}.
\begin{proposition}\label{pro4.2} Suppose that the distribution-dependent stochastic differential equation
\begin{align}\label{dd1}
dX(t)=b(t, X(t), \mathfrak{L}_{X(t)})dt+h(X(t))dW,
\end{align}
has a weak solution $X(t)$ under the probability measure $\mathcal{P}$, let $\mu(t)=\mathfrak{L}_{X(t)}|_{\mathcal{P}}$.  Moreover, if the stochastic differential equation (not distribution-dependent)
\begin{align*}
dX(t)=b(t, X(t), \mu)dt+h(X(t))dW,
\end{align*}
has the pathwise uniqueness for the initial $X(0)$ with $\mathfrak{L}_{X(0)}=\mu(0)$. Then, the equation \eqref{dd1} admits a strong solution with initial data $X(0)$.
\end{proposition}

{\bf Pathwise uniqueness}. For the case of $N=2$, we have the following pathwise uniqueness of equations \eqref{4.3}. We emphasize that for the Allen-Cahn equations, the uniqueness holds for $N=2,3$.
\begin{proposition}\label{pro4.3} Suppose that $\mathbf{u}_1=(\mathbf{u}_{1,*},\mathbf{u}_{1,**}), \mathbf{u}_2=(\mathbf{u}_{2,*}, \mathbf{u}_{2,**})$ are two solutions of equations \eqref{4.3} with the same initial data $\mathbf{u}_1(0) =\mathbf{u}_2(0)$, then the pathwise uniqueness holds in the sense:
$$\mathrm{P}\{\mathbf{u}_1(t)=\mathbf{u}_2(t),~ {\rm for~ all}~ t>0\}=1.$$
\end{proposition}

\begin{proof} Taking $\mathscr{H}=(\varphi,\psi)\in L^2(\Omega;\widetilde{\mathbb{X}})$, the difference of $\mathbf{u}_{1,*}, \mathbf{u}_{2,*}$ satisfies
\begin{align*}
&(\mathbf{u}'_{1,*}-\mathbf{u}'_{2,*}, \varphi)\nonumber\\
=&-\int_{D}\int_{D_\tau}\sum_{i,j=1}^Na_{i,j}\left(\frac{\partial (\mathbf{u}_{1,*}-\mathbf{u}_{2,*})}{\partial x_i}+\frac{\partial (\mathbf{u}_{1,**}-\mathbf{u}_{2,**})}{\partial y_i}\right)\left(\frac{\partial \varphi}{\partial x_j}+\frac{\partial \psi}{\partial y_j}\right)dxdyd\tau\nonumber\\
&-(B(\mathbf{u}_{1,*}, \mathbf{u}_{1,*})-B(\mathbf{u}_{2,*}, \mathbf{u}_{2,*}), \varphi)\nonumber\\
&+(F(\mathbf{u}_{1,*}, \mathfrak{L}_{\mathbf{u}_{1,*}(t)})-F(\mathbf{u}_{2,*}, \mathfrak{L}_{\mathbf{u}_{2,*}(t)}), \varphi)\nonumber\\
&+(G(t, \mathbf{u}_{1,*})-G(t, \mathbf{u}_{2,*}), \varphi)\frac{dW}{dt}.
\end{align*}
Choosing $\varphi=\mathbf{u}_{1,*}-\mathbf{u}_{2,*}, \psi=\mathbf{u}_{1,**}-\mathbf{u}_{2,**}$ in above, by \eqref{2.1*} we have
\begin{align}\label{4.23}
&\frac{d}{dt}\frac{\|\mathbf{u}_{1,*}-\mathbf{u}_{2,*}\|_{H}^2}{2}\nonumber\\
&+\sum_{i,j=1}^N\int_{D}\int_{D_\tau}a_{i,j}(\partial_{x_i} (\mathbf{u}_{1,*}-\mathbf{u}_{2,*})+\partial_{y_i} (\mathbf{u}_{1,**}-\mathbf{u}_{2,**}))\nonumber\\
&\qquad\qquad\qquad\cdot(\partial_{x_j} (\mathbf{u}_{1,*}-\mathbf{u}_{2,*})+\partial_{y_j} (\mathbf{u}_{1,**}-\mathbf{u}_{2,**}))dxdyd\tau\nonumber\\
&=-(B(\mathbf{u}_{1,*}-\mathbf{u}_{2,*}, \mathbf{u}_{1,*}), \mathbf{u}_{1,*}-\mathbf{u}_{2,*})
+(F(\mathbf{u}_{1,*}, \mathfrak{L}_{\mathbf{u}_{1,*}(t)})-F(\mathbf{u}_{2,*}, \mathfrak{L}_{\mathbf{u}_{2,*}(t)}), \mathbf{u}_{1,*}-\mathbf{u}_{2,*})\nonumber\\
&\quad+(G(t, \mathbf{u}_{1,*})-G(t, \mathbf{u}_{2,*}), \mathbf{u}_{1,*}-\mathbf{u}_{2,*})\frac{dW}{dt}\nonumber\\
&\quad+\frac{1}{2}\|G(t, \mathbf{u}_{1,*})-G(t, \mathbf{u}_{2,*})\|_{L_2(H;H)}^2.
\end{align}
Using the fact $\int_{D^N}\frac{\partial \mathbf{u}_{i,**}}{\partial y}dy=0, i=1,2$ (the periodicity of $y$), we have
\begin{align*}
&\int_{D}\int_{D_\tau}a_{i,j}(\partial_{x_i} (\mathbf{u}_{1,*}-\mathbf{u}_{2,*})+\partial_{y_i} (\mathbf{u}_{1,**}-\mathbf{u}_{2,**}))(\partial_{x_j} (\mathbf{u}_{1,*}-\mathbf{u}_{2,*})+\partial_{y_j} (\mathbf{u}_{1,**}-\mathbf{u}_{2,**}))dxdyd\tau\nonumber\\
&=\int_{D}\int_{D_\tau}a_{i,j}\partial_{x_i} (\mathbf{u}_{1,*}-\mathbf{u}_{2,*})\partial_{x_j} (\mathbf{u}_{1,*}-\mathbf{u}_{2,*})dxdyd\tau\nonumber\\
&\quad+\int_{D}\int_{D_\tau}a_{i,j}\partial_{y_i} (\mathbf{u}_{1,**}-\mathbf{u}_{2,**})\partial_{y_j} (\mathbf{u}_{1,**}-\mathbf{u}_{2,**})dxdyd\tau\nonumber\\
&\quad+\int_{D}\int_{D_\tau}a_{i,j}\partial_{x_i} (\mathbf{u}_{1,*}-\mathbf{u}_{2,*})\partial_{y_j} (\mathbf{u}_{1,**}-\mathbf{u}_{2,**})dxdyd\tau\nonumber\\
&\quad+\int_{D}\int_{D_\tau}a_{i,j}\partial_{x_j} (\mathbf{u}_{1,*}-\mathbf{u}_{2,*})\partial_{y_i} (\mathbf{u}_{1,**}-\mathbf{u}_{2,**})dxdyd\tau\nonumber\\
&\geq \int_{D}\int_{D_\tau}a_{i,j}\partial_{x_i} (\mathbf{u}_{1,*}-\mathbf{u}_{2,*})\partial_{x_j} (\mathbf{u}_{1,*}-\mathbf{u}_{2,*})dxdyd\tau\nonumber\\
&\quad+\int_{D}\int_{D_\tau}a_{i,j}\partial_{y_i} (\mathbf{u}_{1,**}-\mathbf{u}_{2,**})\partial_{y_j} (\mathbf{u}_{1,**}-\mathbf{u}_{2,**})dxdyd\tau\nonumber\\
&\quad+\kappa\int_{D}\partial_{x_i} (\mathbf{u}_{1,*}-\mathbf{u}_{2,*})\int_{D_\tau}\partial_{y_j} (\mathbf{u}_{1,**}-\mathbf{u}_{2,**})dxdyd\tau\nonumber\\
&\quad+\kappa\int_{D}\partial_{x_j} (\mathbf{u}_{1,*}-\mathbf{u}_{2,*})\int_{D_\tau}\partial_{y_i} (\mathbf{u}_{1,**}-\mathbf{u}_{2,**})dxdyd\tau\nonumber\\
&=\int_{D}\int_{D_\tau}a_{i,j}\partial_{x_i} (\mathbf{u}_{1,*}-\mathbf{u}_{2,*})\partial_{x_j} (\mathbf{u}_{1,*}-\mathbf{u}_{2,*})dxdyd\tau\nonumber\\
&\quad+\int_{D}\int_{D_\tau}a_{i,j}\partial_{y_i} (\mathbf{u}_{1,**}-\mathbf{u}_{2,**})\partial_{y_j} (\mathbf{u}_{1,**}-\mathbf{u}_{2,**})dxdyd\tau,
\end{align*}
which along with assumption \eqref{1.2} gives that
\begin{align}
&\sum_{i,j=1}^N\int_{D}\int_{D_\tau}a_{i,j}(\partial_{x_i} (\mathbf{u}_{1,*}-\mathbf{u}_{2,*})+\partial_{y_i} (\mathbf{u}_{1,**}-\mathbf{u}_{2,**}))\nonumber\\
&\qquad\qquad\qquad\cdot(\partial_{x_j} (\mathbf{u}_{1,*}-\mathbf{u}_{2,*})+\partial_{y_j} (\mathbf{u}_{1,**}-\mathbf{u}_{2,**}))dxdyd\tau\nonumber\\
&\geq \kappa\|\mathbf{u}_1-\mathbf{u}_2\|_{\mathbb{X}}^2.
\end{align}
By \eqref{2.4} and the Young inequality, we have
\begin{align}
-(B(\mathbf{u}_{1,*}-\mathbf{u}_{2,*}, \mathbf{u}_{1,*}), \mathbf{u}_{1,*}-\mathbf{u}_{2,*})&\leq C\|\mathbf{u}_{1,*}-\mathbf{u}_{2,*}\|_{H}\|\mathbf{u}_{1,*}-\mathbf{u}_{2,*}\|_{V}\|\mathbf{u}_{1,*}\|_{V}\nonumber\\
&\leq C\|\mathbf{u}_{1,*}\|_{V}^2\|\mathbf{u}_{1,*}-\mathbf{u}_{2,*}\|_{H}^2+\frac{\kappa}{2}\|\mathbf{u}_{1,*}-\mathbf{u}_{2,*}\|_{V}^2.
\end{align}
The assumption {\bf (F.3)} leads to
\begin{align}\label{4.26}
&(F(\mathbf{u}_{1,*}, \mathfrak{L}_{\mathbf{u}_{1,*}(t)})-F(\mathbf{u}_{2,*}, \mathfrak{L}_{\mathbf{u}_{2,*}(t)}), \mathbf{u}_{1,*}-\mathbf{u}_{2,*})\nonumber\\
&\leq C(1+\mathrm{E}\|\mathbf{u}_{1,*}\|^a_{H}+\mathrm{E}\|\mathbf{u}_{2,*}\|^a_{H})
\|\mathbf{u}_{1,*}-\mathbf{u}_{2,*}\|_{H}^2\nonumber\\&\quad+\ell \mathcal{W}_{2, H}(\mathfrak{L}_{\mathbf{u}_{1,*}(t)}, \mathfrak{L}_{\mathbf{u}_{2,*}(t)})+\eta\|\mathbf{u}_{1,*}-\mathbf{u}_{2,*}\|_{V}^2\nonumber\\
&\leq C(1+\mathrm{E}\|\mathbf{u}_{1,*}\|^a_{H}+\mathrm{E}\|\mathbf{u}_{2,*}\|^a_{H})
\|\mathbf{u}_{1,*}-\mathbf{u}_{2,*}\|_{H}^2
\nonumber\\&\quad+\ell\mathrm{E}\|\mathbf{u}_{1,*}-\mathbf{u}_{2,*}\|_{H}^2+\eta\|\mathbf{u}_{1,*}-\mathbf{u}_{2,*}\|_{V}^2.
\end{align}

Considering \eqref{4.23}-\eqref{4.26}, we obtain
\begin{align}\label{4.27}
&\frac{d}{dt}\|\mathbf{u}_{1,*}-\mathbf{u}_{2,*}\|_{H}^2+(\kappa-2\eta)\|\mathbf{u}_1-\mathbf{u}_2\|_{\mathbb{X}}^2\nonumber\\
&\leq C(\|\mathbf{u}_{1,*}\|_{V}^2+
\mathrm{E}\|\mathbf{u}_{1,*}\|^a_{H}+\mathrm{E}\|\mathbf{u}_{2,*}\|^a_{H}+1)
\|\mathbf{u}_{1,*}-\mathbf{u}_{2,*}\|_{H}^2\nonumber\\&\quad+2\ell\mathrm{E}\|\mathbf{u}_{1,*}-\mathbf{u}_{2,*}\|_{H}^2
+2(G( t, \mathbf{u}_{1,*})-G( t, \mathbf{u}_{2,*}), \mathbf{u}_{1,*}-\mathbf{u}_{2,*})\frac{dW}{dt}.
\end{align}

Define the function
$$J(t)=C(\mathrm{E}\|\mathbf{u}_{1,*}\|^a_{H}+\mathrm{E}\|\mathbf{u}_{2,*}\|^a_{H}+\|\mathbf{u}_{1,*}\|_{V}^2+1),$$
then using the It\^o product formula to the function
$${\rm exp}\left(\int_{0}^{t}-J(s)ds\right)\left(\|\mathbf{u}_{1,*}-\mathbf{u}_{2,*}\|_{H}^2\right),$$
by \eqref{4.27} we have
\begin{align}\label{4.33}
&d{\rm exp}\left(\int_{0}^{t}-J(s)ds\right)\left(\|\mathbf{u}_{1,*}-\mathbf{u}_{2,*}\|_{H}^2\right)\nonumber\\
&={\rm exp}\left(\int_{0}^{t}-J(s)ds\right)d\|\mathbf{u}_{1,*}-\mathbf{u}_{2,*}\|_{H}^2-J(s){\rm exp}\left(\int_{0}^{t}-J(s)ds\right)\left(\|\mathbf{u}_{1,*}-\mathbf{u}_{2,*}\|_{H}^2\right)dt\nonumber\\
&\leq {\rm exp}\left(\int_{0}^{t}-J(s)ds\right)\left(-(\kappa-2\eta)\|\mathbf{u}_1-\mathbf{u}_2\|_{\mathbb{X}}^2
+2\ell\mathrm{E}\|\mathbf{u}_{1,*}-\mathbf{u}_{2,*}\|_{H}^2\right)dt\nonumber\\
&\quad +2{\rm exp}\left(\int_{0}^{t}-J(s)ds\right)(G(t, \mathbf{u}_{1,*})-G(t, \mathbf{u}_{2,*}), \mathbf{u}_{1,*}-\mathbf{u}_{2,*})dW.
\end{align}
From \eqref{1.23**}, we know for any $t>0$
$$0<\int_{0}^{t}J(s)ds<\infty,~ {\rm P} ~{\rm a.s.}$$
then, integrating of time and taking expectation, by \eqref{4.33} we have
\begin{align*}
&\mathrm{E}\left({\rm exp}\left(\int_{0}^{t}-J(s)ds\right)\left(\|\mathbf{u}_{1,*}-\mathbf{u}_{2,*}\|_{H}^2\right)\right)\nonumber\\
&\leq\mathrm{E} \int_{0}^{t}{\rm exp}\left(\int_{0}^{s}-J(r)dr\right)\left(-(\kappa-2\eta)\|\mathbf{u}_1-\mathbf{u}_2\|_{\mathbb{X}}^2+2\ell\mathrm{E}\|\mathbf{u}_{1,*}-\mathbf{u}_{2,*}\|_{H}^2\right)ds\nonumber\\
&\leq \mathrm{E} \int_{0}^{t}\left(-(\kappa-2\eta)\|\mathbf{u}_1-\mathbf{u}_2\|_{\mathbb{X}}^2+2\ell\mathrm{E}\|\mathbf{u}_{1,*}-\mathbf{u}_{2,*}\|_{H}^2\right)ds\nonumber\\
&\leq 0.
\end{align*}
Note that we also used the martingale property to cancel the noise part.

Finally, we get $\mathbf{u}_{1,*}=\mathbf{u}_{2,*}$, P a.s. Furthermore, using \eqref{4.27}, we have $\mathbf{u}_{1,**}=\mathbf{u}_{2,**}$, P a.s. This completes the proof.
\end{proof}

Invoked by Proposition \ref{pro4.2}, we also require the strong uniqueness of the following decoupled stochastic abstract fluid models
\begin{align}\label{4.30}
(\mathbf{u}'(t), \varphi)=&-\sum_{i,j=1}^2\int_{D}\int_{D_\tau}a_{i,j}\left(\frac{\partial \mathbf{u}}{\partial x_i}+\frac{\partial \overline{\mathbf{u}}}{\partial y_i}\right)\left(\frac{\partial \varphi}{\partial x_j}+\frac{\partial \psi}{\partial y_j}\right)dxdyd\tau-(B(\mathbf{u}, \mathbf{u}), \varphi)\nonumber\\
&+(F^\mu(\mathbf{u}), \varphi)+(G(t, \mathbf{u}), \varphi)\frac{dW}{dt},
\end{align}
where $F^\mu(\mathbf{u})$ is denoted by $F(\mathbf{u}, \mu)$.

Using the assumptions imposed on $B, F, G$ and the coefficients $a_{i,j}$, by a same argument as Proposition \ref{pro4.3}, we also have the strong uniqueness of equations \eqref{4.30}. Then, the modified Yamada-Watanabe theorem implies that the solution $(\mathbf{u}, \overline{\mathbf{u}})$ of equations \eqref{4.3} is unique, which is strong in the sense of probability.

{\bf Recover the representation of $\overline{\mathbf{u}}$}. We are going to give the specific expression of the function $\overline{\mathbf{u}}$.

\begin{lemma} The extension part $\overline{\mathbf{u}}$ is given by
\begin{align*}
\overline{\mathbf{u}}(x,t,y,\tau)=-\sum_{i,k=1}^N \frac{\partial \mathbf{u}}{\partial x_i}(x,t)\eta_{i,k}(y, \tau),~ \mathrm{P} ~a.s.
\end{align*}
where $\eta_{i,k}$ is the solution of variational problem
\begin{eqnarray}\label{var1}
\rho(\eta_{i,k}, \mathbf{w})=\sum_{j=1}^N \int_{D_\tau}a_{i,j}\frac{\partial\mathbf{w}^k}{\partial y_j}dy d\tau,
\end{eqnarray}
for any $\mathbf{w}\in V_{per}$, $\mathrm{P}\times \mathcal{G}, a.e. ~(\omega, y, \tau)$, $\mathcal{G}$ is the Lebesgue measure and the bilinear operator $\rho$ is defined by
$$\rho(\mathbf{u}, \mathbf{v})=\sum_{i,j=1}^N\int_{D_\tau}a_{i,j}\frac{\partial \mathbf{u}}{\partial y_i}\frac{\partial \mathbf{v}}{\partial y_j}dyd\tau.$$
\end{lemma}
\begin{proof}  Choosing $\varphi=0$ in equations \eqref{4.3}, and $\psi=\phi \mathbf{w}$ for $\phi\in C^\infty(D_t\times \widetilde{T})$ and $ \mathbf{w}\in V_{per}$, we have
\begin{align}\label{v3}
\rho(\overline{\mathbf{u}}, \mathbf{w})=\sum_{i, j, k=1}^N \int_{D_t} \frac{\partial\mathbf{u}}{\partial x_i}\int_{D_\tau}a_{i,j}\frac{\partial\mathbf{w}^k}{\partial y_j}dy d\tau dxdt, ~ \mathrm{P} ~a.s.
\end{align}
According to the well-posedness of the variational problem \eqref{var1}, the process $\overline{\mathbf{u}}$ is the unique solution to \eqref{var1}. Compared with \eqref{var1}, we find that
$$\mathbf{f}=-\sum_{i,k=1}^N \frac{\partial \mathbf{u}}{\partial x_i}(x,t)\eta_{i,k}(y, \tau),  ~ \mathrm{P} ~a.s.$$
is also the solution of \eqref{v3}. The uniqueness implies that $\overline{\mathbf{u}}=\mathbf{f}$.
\end{proof}

Denote the function
$$\mathbf{a}_{i,j,k,l}=\int_{D_\tau}a_{i,j}dy d\tau+\int_{D_{\tau}}a_{i,j}\frac{\partial\eta^l_{i,k}}{\partial y_j}dy d\tau,$$
for $1\leq i,j,k,l \leq N$, corresponding to the function, we denote by $\widetilde{A}=(\widetilde{A}_{kl})_{k,l=1,\cdots, N}$ the differential homogenized operator
$$\widetilde{A}_{kl}=-\sum_{i,j=1}^N\mathbf{a}_{i,j,k,l}\frac{\partial^2}{\partial x_i \partial y_j},~~ k,l=1,2,\cdots,N.$$

Observe that the homogenized operator $\widetilde{A}$ is the uniformly elliptic see \cite{ben}, thus, there exists a constant $\kappa$ such that
$$\sum_{i,j,k,l=1}^N\mathbf{a}_{i,j,k,l}\xi_{i,k}\xi_{j,l}\geq \kappa \sum_{k,l=1}^N|\xi_{k,l}|^2.$$
Using the uniformly elliptic condition, assumptions ({\bf F.1})-({\bf F.3}) and ({\bf G.1})-({\bf G.2}), we could infer that by a same argument as previously, the equations
\begin{eqnarray*}
\partial_t \mathbf{u}+\widetilde{A }\mathbf{u}+B(\mathbf{u}, \mathbf{u})=F(\mathbf{u}, \mathfrak{L}_{\mathbf{u}(t)})+G(t, \mathbf{u})\frac{dW}{dt}
\end{eqnarray*}
have the solution $\mathbf{u}\in L^p(\Omega; L^\infty(0,T;H)\cap L^2(0,T;V)\cap L^4(D\times (0,T)))$ for any $p\geq 2$. When $N=2$, the solution is pathwise unique. Particularly, the pathwise uniqueness of solution to the Allen-Cahn equations holds for both $N=2,3$.

\section{A corrector result}
In this section, we establish a corrector result which improves the convergence of $\nabla \mathbf{u}^\varepsilon$ in $L^2(\Omega\times D_t),~\Sigma\!-\!{\rm weak}$ to the $L^2(\Omega\times D_t),~\Sigma\!-\!{\rm strong}$.
\begin{proposition} Under the same assumptions as those of in Proposition \ref{pro2.1}, we have
$$\frac{\partial\mathbf{u}^\varepsilon}{\partial x_i}\rightarrow \frac{\partial\mathbf{u}}{\partial x_i}+\frac{\partial\overline{\mathbf{u}}}{\partial y_i}, ~{\rm in }~ L^2(\Omega\times D_t),~\Sigma\!-\!{\rm strong}, ~1\leq i\leq N$$
as $\varepsilon\rightarrow 0$.
\end{proposition}

\begin{proof}
From \eqref{a2.14}, we have
\begin{align*}
\int_0^T(A^\varepsilon \mathbf{u}^\varepsilon, \mathbf{u}^\varepsilon)_{V'\times V}dt=&-\frac{1}{2}\|\mathbf{u}^\varepsilon\|_H^2+\frac{1}{2}\|\mathbf{u}_0\|_H^2+\int_{0}^{T}(F(\mathbf{u}^\varepsilon, \mathfrak{L}_{\mathbf{u}^\varepsilon(t)}), \mathbf{u}^\varepsilon)dt\nonumber\\&+ \int_{0}^{T}(G(t, \mathbf{u}^\varepsilon)dW, \mathbf{u}^\varepsilon)+\frac{1}{2}\int_{0}^{T}\|G(t, \mathbf{u}^\varepsilon)\|_{L_2(H,H)}^2 dt.
\end{align*}
Taking expectation on both sides, we see
\begin{align}\label{2.14*}
\mathrm{E}\int_0^T(A^\varepsilon \mathbf{u}^\varepsilon, \mathbf{u}^\varepsilon)_{V'\times V}dt=&-\frac{1}{2}\mathrm{E}\|\mathbf{u}^\varepsilon\|_H^2+\frac{1}{2}\|\mathbf{u}_0\|_H^2+\mathrm{E}\int_{0}^{T}(F(\mathbf{u}^\varepsilon, \mathfrak{L}_{\mathbf{u}^\varepsilon(t)}), \mathbf{u}^\varepsilon)dt\nonumber\\&+\frac{1}{2}\mathrm{E}\int_{0}^{T}\|G(t, \mathbf{u}^\varepsilon)\|_{L_2(H,H)}^2 dt,
\end{align}
where we used the fact that the stochastic integral term is a square integrable martingale which the mean value is zero. By \eqref{2.14*}, we further have
\begin{align}\label{6.3}
\mathrm{E}\int_0^T(A^\varepsilon (\mathbf{u}^\varepsilon-\Psi^\varepsilon), \mathbf{u}^\varepsilon-\Psi^\varepsilon)_{V'\times V}dt=&-\frac{1}{2}\mathrm{E}\|\mathbf{u}^\varepsilon\|_H^2+\frac{1}{2}\|\mathbf{u}_0\|_H^2
+\mathrm{E}\int_{0}^{T}(F(\mathbf{u}^\varepsilon, \mathfrak{L}_{\mathbf{u}^\varepsilon(t)}), \mathbf{u}^\varepsilon)dt\nonumber\\&+\frac{1}{2}\mathrm{E}\int_{0}^{T}\|G(t, \mathbf{u}^\varepsilon)\|_{L_2(H,H)}^2 dt-2\mathrm{E}\int_0^T(A^\varepsilon \mathbf{u}^\varepsilon, \Psi^\varepsilon)_{V'\times V}dt\nonumber\\
&+\mathrm{E}\int_0^T(A^\varepsilon \Psi^\varepsilon, \Psi^\varepsilon)_{V'\times V}dt,
\end{align}
where $\Psi^\varepsilon$ is defined as that of in Proposition \ref{pro2.1}.

As the proof of Proposition \ref{pro2.1}, we know
\begin{align}\label{5.4}
&\mathrm{E}\int_{0}^{T}(F(\mathbf{u}^\varepsilon, \mathfrak{L}_{\mathbf{u}^\varepsilon(t)}), \mathbf{u}^\varepsilon)dt\rightarrow\mathrm{E}\int_{0}^{T}(F(\mathbf{u}, \mathfrak{L}_{\mathbf{u}(t)}), \mathbf{u})dt,\nonumber\\
&\mathrm{E}\int_0^T(A^\varepsilon \mathbf{u}^\varepsilon, \Psi^\varepsilon)_{V'\times V}dt\rightarrow \sum_{i,j=1}^N \mathrm{E}\int_{D_t}\int_{D_\tau}a_{i,j}\left(\frac{\partial \mathbf{u}}{\partial x_i}+\frac{\partial \overline{\mathbf{u}}}{\partial y_i}\right)\left(\frac{\partial \varphi}{\partial x_j}+\frac{\partial \psi}{\partial y_j}
\right)1_{\mathcal{A}}(\omega)dxdydtd\tau,\\
&\mathrm{E}\int_0^T(A^\varepsilon \Psi^\varepsilon, \Psi^\varepsilon)_{V'\times V}dt\rightarrow \sum_{i,j=1}^N \mathrm{E}\int_{D_t}\int_{D_\tau}a_{i,j}\left(\frac{\partial \varphi}{\partial x_i}+\frac{\partial \psi}{\partial y_i}\right)\left(\frac{\partial \varphi}{\partial x_j}+\frac{\partial \psi}{\partial y_j}
\right)1_{\mathcal{A}}(\omega)dxdydtd\tau.\nonumber
\end{align}
From \eqref{2.1}, for a.e. $t$, we obtain $\mathbf{u}^\varepsilon\rightarrow \mathbf{u}$ in $H$,  the dominated convergence theorem leads to
\begin{align}\label{5.5}
\mathrm{E}\|\mathbf{u}^\varepsilon\|_H^2\rightarrow\mathrm{E}\|\mathbf{u}\|_H^2.
\end{align}
By assumption $({\bf G.1})$ and \eqref{2.1}, we get $\mathrm{P}$ a.s.
\begin{align*}
\int_{0}^{T}\|G(t, \mathbf{u}^\varepsilon)\|_{L_2(H,H)}^2 dt\rightarrow \int_{0}^{T}\|G(t, \mathbf{u})\|_{L_2(H,H)}^2 dt,
\end{align*}
then, we infer from \eqref{1.23} and the dominated convergence theorem
\begin{align}\label{5.6}
\mathrm{E}\int_{0}^{T}\|G(t, \mathbf{u}^\varepsilon)\|_{L_2(H,H)}^2 dt\rightarrow \mathrm{E}\int_{0}^{T}\|G(t, \mathbf{u})\|_{L_2(H,H)}^2 dt.
\end{align}
Combining \eqref{6.3}-\eqref{5.6} and Proposition \ref{pro2.1}, we find
\begin{align}\label{5.7}
&\lim_{\varepsilon\rightarrow 0}\mathrm{E}\int_0^T(A^\varepsilon (\mathbf{u}^\varepsilon-\Psi^\varepsilon), \mathbf{u}^\varepsilon-\Psi^\varepsilon)_{V'\times V}dt\nonumber\\
&=-\frac{1}{2}\mathrm{E}\|\mathbf{u}\|_H^2+\frac{1}{2}\|\mathbf{u}_0\|_H^2+\mathrm{E}\int_{0}^{T}(F(\mathbf{u}, \mathfrak{L}_{\mathbf{u}(t)}), \mathbf{u})dt+\frac{1}{2}\mathrm{E}\int_{0}^{T}\|G(t, \mathbf{u})\|_{L_2(H,H)}^2 dt\nonumber\\
&\quad-2\sum_{i,j=1}^N \mathrm{E}\int_{D_t}\int_{D_\tau}a_{i,j}\left(\frac{\partial \mathbf{u}}{\partial x_i}+\frac{\partial \overline{\mathbf{u}}}{\partial y_i}\right)\left(\frac{\partial \varphi}{\partial x_j}+\frac{\partial \psi}{\partial y_j}
\right)1_{\mathcal{A}}(\omega)dxdydtd\tau\nonumber\\
&\quad+\sum_{i,j=1}^N \mathrm{E}\int_{D_t}\int_{D_\tau}a_{i,j}\left(\frac{\partial \varphi}{\partial x_i}+\frac{\partial \psi}{\partial y_i}\right)\left(\frac{\partial \varphi}{\partial x_j}+\frac{\partial \psi}{\partial y_j}
\right)1_{\mathcal{A}}(\omega)dxdydtd\tau\nonumber\\
&=\sum_{i,j=1}^N \mathrm{E}\int_{D_t}\int_{D_\tau}a_{i,j}\left(\frac{\partial \mathbf{u}}{\partial x_i}+\frac{\partial \overline{\mathbf{u}}}{\partial y_i}\right)\left(\frac{\partial \mathbf{u}}{\partial x_j}+\frac{\partial \overline{\mathbf{u}}}{\partial y_j}
\right)dxdydtd\tau\nonumber\\
&\quad-2\sum_{i,j=1}^N \mathrm{E}\int_{D_t}\int_{D_\tau}a_{i,j}\left(\frac{\partial \mathbf{u}}{\partial x_i}+\frac{\partial \overline{\mathbf{u}}}{\partial y_i}\right)\left(\frac{\partial \varphi}{\partial x_j}+\frac{\partial \psi}{\partial y_j}
\right)1_{\mathcal{A}}(\omega)dxdydtd\tau\nonumber\\
&\quad+\sum_{i,j=1}^N \mathrm{E}\int_{D_t}\int_{D_\tau}a_{i,j}\left(\frac{\partial \varphi}{\partial x_i}+\frac{\partial \psi}{\partial y_i}\right)\left(\frac{\partial \varphi}{\partial x_j}+\frac{\partial \psi}{\partial y_j}
\right)1_{\mathcal{A}}(\omega)dxdydtd\tau\nonumber\\
&=:\mathrm{E}\int_0^T(\widetilde{a} (\widetilde{\mathbf{u}}-\Psi), \widetilde{\mathbf{u}}-\Psi)dt,
\end{align}
where $\widetilde{\mathbf{u}}=(\mathbf{u}, \overline{\mathbf{u}})$, $\Psi=(\varphi1_{\mathcal{A}}(\omega), \psi1_{\mathcal{A}}(\omega))$, $\varphi\in C^\infty(D_t), \psi\in C^\infty(D_t)\times C^\infty_{per}(D_\tau)$ are chosen to satisfy that
\begin{align}\label{5.8}
\mathrm{E}\int_0^T(\widetilde{a} (\widetilde{\mathbf{u}}-\Psi), \widetilde{\mathbf{u}}-\Psi)dt\leq \varepsilon_1,
\end{align}
 $\varepsilon_1$ is a constant that could be small arbitrarily. This could be achieved due to the fact that $C^\infty(D_t)\times 1_{\cdot}(\omega)$ and $C^\infty(D_t)\times C^\infty_{per}(D_\tau)\times 1_{\cdot}(\omega)$ are dense in $L^2(\Omega\times D_t), L^2(\Omega \times D_t; L^2_{per}(D_\tau))$ respectively. By \eqref{5.7} and \eqref{5.8}, we infer that there exists  $\eta>0$ such that for all $\varepsilon<\eta$
\begin{align*}
\mathrm{E}\int_0^T(A^\varepsilon (\mathbf{u}^\varepsilon-\Psi^\varepsilon), \mathbf{u}^\varepsilon-\Psi^\varepsilon)_{V'\times V}dt\leq 2\varepsilon_1.
\end{align*}
By condition \eqref{1.2}, we further obtain
\begin{align}\label{5.9}
\mathrm{E}\int_0^T(\mathbf{u}^\varepsilon-\Psi^\varepsilon, \mathbf{u}^\varepsilon-\Psi^\varepsilon)_{V}dt\leq \frac{2}{\kappa}\varepsilon_1,
\end{align}
also,
\begin{align}\label{5.88}
\mathrm{E}\int_0^T(\widetilde{\mathbf{u}}-\Psi, \widetilde{\mathbf{u}}-\Psi)dt\leq \frac{\varepsilon_1}{\kappa}.
\end{align}

With the preliminary inequalities in hands,  we next prove the result using Lemma \ref{lem4}, it remains to show that
$$\left\|\frac{ \partial\mathbf{u}^\varepsilon}{\partial x_i}\right\|_{L^2(\Omega\times D_t)}\rightarrow \left\|\frac{\partial\mathbf{u}}{\partial x_i}+\frac{\partial\overline{\mathbf{u}}}{\partial y_i}\right\|_{L^2(\Omega\times D_t, L^2_{per}(D_{\tau}))}.$$

First, note that
\begin{align*}
\frac{ \partial\Psi^\varepsilon}{\partial x_i}\rightarrow \left(\frac{\partial\varphi}{\partial x_i}+\frac{\partial\psi}{\partial y_i}\right)1_{\mathcal{A}}(\omega), ~{\rm in }~ L^2(\Omega\times D_t),~\Sigma\!-\!{\rm strong},
\end{align*}
then
\begin{align*}
\left\|\frac{ \partial\Psi^\varepsilon}{\partial x_i}\right\|_{L^2(\Omega\times D_t)}\rightarrow\left\|\left(\frac{\partial\varphi}{\partial x_i}+\frac{\partial\psi}{\partial y_i}\right)1_{\mathcal{A}}(\omega)\right\|_{L^2(\Omega\times D_t, L^2_{per}(D_{\tau}))}.
\end{align*}
Thus, for any $\varepsilon_2>0$, there exists  $\delta>0$ such that $\varepsilon<\delta$, we have
\begin{align}\label{5.11}
\left|\left\|\frac{ \partial\Psi^\varepsilon}{\partial x_i}\right\|_{L^2(\Omega\times D_t)}-\left\|\left(\frac{\partial\varphi}{\partial x_i}+\frac{\partial\psi}{\partial y_i}\right)1_{\mathcal{A}}(\omega)\right\|_{L^2(\Omega\times D_t, L^2_{per}(D_{\tau}))}\right|\leq \varepsilon_2.
\end{align}

Using the triangle inequality and \eqref{5.9}-\eqref{5.11}, we conclude
\begin{align*}
&\left|\left\|\frac{ \partial\mathbf{u}^\varepsilon}{\partial x_i}\right\|_{L^2(\Omega\times D_t)}-\left\|\frac{\partial\mathbf{u}}{\partial x_i}+\frac{\partial\overline{\mathbf{u}}}{\partial y_i}\right\|_{L^2(\Omega\times D_t, L^2_{per}(D_{\tau}))}\right|\nonumber\\
&\leq \left|\left\|\frac{ \partial\mathbf{u}^\varepsilon}{\partial x_i}\right\|_{L^2(\Omega\times D_t)}-\left\|\frac{ \partial\Psi^\varepsilon}{\partial x_i}\right\|_{L^2(\Omega\times D_t)}\right|\nonumber\\
&\quad+\left|\left\|\frac{ \partial\Psi^\varepsilon}{\partial x_i}\right\|_{L^2(\Omega\times D_t)}-\left\|\left(\frac{\partial\varphi}{\partial x_i}+\frac{\partial\psi}{\partial y_i}\right)1_{\mathcal{A}}(\omega)\right\|_{L^2(\Omega\times D_t, L^2_{per}(D_{\tau}))}\right|\nonumber\\
&\quad+\left|\left\|\frac{\partial\mathbf{u}}{\partial x_i}+\frac{\partial\overline{\mathbf{u}}}{\partial y_i}\right\|_{L^2(\Omega\times D_t, L^2_{per}(D_{\tau}))}-\left\|\left(\frac{\partial\varphi}{\partial x_i}+\frac{\partial\psi}{\partial y_i}\right)1_{\mathcal{A}}(\omega)\right\|_{L^2(\Omega\times D_t, L^2_{per}(D_{\tau}))}\right|\nonumber\\
&\leq \frac{3}{\kappa}\varepsilon_1+\varepsilon_2.
\end{align*}
We complete the proof following the arbitrariness of $\varepsilon_1, \varepsilon_2$.
\end{proof}

\section{Appendix}
In the appendix, we introduce two lemmas used frequently in this paper. In order to establish the tightness of a family of probability measures, we first give the Aubin-Lions lemma.
For any $p\geq 1$, denote by
 \begin{equation*}
W^{1,p}(0,T;X):=\left\{\mathbf{u}\in L^{p}(0,T;X):\frac{d\mathbf{u}}{dt}\in L^{p}(0,T;X)\right\},
\end{equation*}
which is the classical Sobolev space with its usual norm,
\begin{equation*}
\|\mathbf{u}\|_{W^{1,p}(0,T;X)}^{p}=\int_{0}^{T}\|\mathbf{u}(t)\|_{X}^{p}+\left\|\frac{d\mathbf{u}}{dt}(t)\right\|_{X}^{p}dt.
\end{equation*}
\begin{lemma}\cite[Corollary 5]{Simon}\label{Simon1} Suppose that $X_{1}\subset X_{0}\subset X_{2}$ are Banach spaces and $X_{1}$ and $X_{2}$ are reflexive and the embedding of $X_{1}$ into $X_{0}$ is compact.
Then for any $1\leq p<\infty$, the embedding
\begin{equation*}
L^{p}(0,T; X_{1})\cap W^{1,2}(0,T; X_{2})\hookrightarrow L^{p}(0,T; X_{0}),
\end{equation*}
is compact.
\end{lemma}

The following Vitali convergence theorem is applied to identify the limit.
\begin{theorem}\cite[Chapter 3]{Kallenberg}\label{thm4.1} Let $p\geq 1$, $\{\mathbf{u}_n\}_{n\geq 1}\in L^p$ and $\mathbf{u}_n\rightarrow \mathbf{u}$ in probability. Then, the followings are equivalent:\\
i. $\mathbf{u}_n\rightarrow \mathbf{u}$ in $L^p$;\\
ii. the sequence $|\mathbf{u}_n|^p$ is uniformly integrable;\\
iii. $\mathrm{E}|\mathbf{u}_n|^p\rightarrow \mathrm{E}|\mathbf{u}|^p$.
\end{theorem}

\section*{Acknowledgments}
The authors would like to thank the anonymous referees for their very valuable suggestions and constructive comments. Y. Tang is supported by the National Natural Science Foundation of China
(Grant No. 12171442). Z. Qiu is supported by the National Natural Science Foundation of China
(Grant No. 12401305), the National Science Foundation for Colleges and Universities in
Jiangsu Province (Grant No. 24KJB110011) and the National Science Foundation of Jiangsu
Province (Grant No. BK20240721).

\section*{Data availability}
Data sharing not applicable to this article as no datasets were generated or analysed
during the current study.

\section*{Statements and Declarations}
The authors have no relevant financial or non-financial interests to disclose.


\smallskip



\bigskip
\begin{thebibliography}{aa}
\footnotesize















\bibitem{all}
G. Allaire, Homogenization and two-scale convergence, SIAM J. Math. Anal. 23 (1992)
 1482–1518.

\bibitem{all2}
 G. Allaire, A. Piatnitski,  Homogenization of nonlinear reaction-diffusion equation with a large reaction term, Ann.
Univ. Ferrara 56 (2010) 141–161.

\bibitem{all3}
 G. Allaire, A. Mikeli\'{c}, A. Piatnitski,  Homogenization approach to the dispersion theory for reactive transport
through porous media, SIAM J. Math. Anal. 42 (2010) 125–144.

\bibitem{ben2}
A. Bensoussan,  Homogenization of a class of stochastic partial differential equations, Prog. Nonlinear Differ. Equ.
Appl. 5 (1991) 47–65.

\bibitem{ben}
  A. Bensoussan,  J.L. Lions,  G. Papanicolaou,  Asymptotic analysis for periodic structures, North-Holland, Amsterdam, 1978.

\bibitem{brz}
Z. Brze\'{z}niak, M. Ondrej\'{a}t,  Stochastic geometric wave equations with values
in compact Riemannian homogeneous spaces, Ann. Probab. 41 (2020)  1938–1977.


\bibitem{CD}
 R. Carmona,   F. Delarue,  Probabilistic theory of mean field games with
applications, Vol. I and II, Springer International Publishing, 2018.

\bibitem{CT}
J. Chen, Y. Tang,  Homogenization of non-local nonlinear p-Laplacian equation with variable index and periodic structure, J. Math. Phys. 64(6) (2023) 061502.



\bibitem{Zabczyk} G. Da Prato, J. Zabczyk,  Stochastic equations in infinite dimensions, Cambridge University Press, 1992.

\bibitem{HHL}W. Hong, S. Hu, W. Liu,  McKean-Vlasov SDEs and SPDEs with locally monotone
coefficients,  Ann. App. Probab. 34  (2024) 2136–2189.


\bibitem{HJS2}
Q. Huang, J. Duan, R. Song, Homogenization of nonlocal partial differential equations related to stochastic differential equations with L\'{e}vy noise, Bernoulli 28(3)  (2022) 1648–1674.

\bibitem{HJS}
Q. Huang, J. Duan, R. Song, Homogenization of non-symmetric jump processes, Adv. Appl. Probab. 56(1) (2024) 1–33.

\bibitem{HW}
 X. Huang, F.Y. Wang,   McKean-Vlasov SDEs with drifts discontinuous under
Wasserstein distance, Discre. Contin. Dyn. Sys. 41 (2021) 1667–1679.


\bibitem{MeanHuang}
Z. Huang, S. Tang, Mean field games with major and minor agents: the limiting problem and Nash equilibrium, Stochastics  (2024) 1–34.


\bibitem{ich}
N. Ichihara,  Homogenization for stochastic partial differential equations derived from nonlinear filterings with
feedback, J. Math. Soc. Japan 57 (2005) 593–603.

\bibitem{Kallenberg}
O. Kallenberg,  Foundations of modern probability: second edition, New York, Springer, 2002.

\bibitem{mck}
H.P. McKean,   A class of Markov processes associated with nonlinear parabolic equations, Proc. Nat. Acad. Sci. U.S.A. 56 (1966) 1907–1911.

\bibitem{moh}
 M. Mohammed, M. Sango,  Homogenization of linear hyperbolic stochastic partial
differential equation with rapidly oscillating coefficients: the two scale convergence
method, Asymptotic Anal. 91  (2015) 341–371.


\bibitem{moh2}
 M. Mohammed,  Homogenization and correctors for linear stochastic
equations via the periodic unfolding methods, Stoch. Dyn.
19(2) (2019) 1950040.

\bibitem{Allen2022}
 P.S. Morfe, Homogenization of the Allen–Cahn equation with periodic mobility, Cal. Var. Partial Differential Equations,  61(3) (2022) 110.







\bibitem{ngu2}
G. Nguetseng,  A general convergence result for a functional related to the theory of
homogenization, SIAM J. Math. Anal. 20 (1989) 608–623.

\bibitem{ngu}
 G. Nguetseng,  L. Signing,  Sigma-convergence of stationary Navier–Stokes type equations, Electron. J.
Differ. Equ. 2009(74) (2009) 1–18.




\bibitem{par}
 E. Pardoux, A.L. Piatnitski,  Homogenization of a nonlinear random parabolic partial differential equation, Stoch. Proc. Appl. 104  (2003) 1–27.

\bibitem{raz}
P.A. Razafimandimby, M. Sango, J.L. Woukeng,  Homogenization of a stochastic
nonlinear reaction–diffusion equation with a large reaction term: the almost periodic
framework, J. Math. Anal. Appl. 394  (2012) 186–212.

\bibitem{raz2}
 P.A. Razafimandimby, J.L. Woukeng, Homogenization of nonlinear stochastic
partial differential equations in a general ergodic environment, Stoch. Anal. Appl. 31 (2013)
 755–784.

\bibitem{san3}
 M. Sango,  Asymptotic behavior of a stochastic evolution problem in a varying domain,
Stoch. Anal. Appl. 20 (2002) 1331–1358.

\bibitem{san}
M. Sango, Homogenization of stochastic semilinear parabolic equations with non-Lipschitz forcings in domains with fine grained boundaries, Commun. Math. Sci. 12  (2014)
 345–382.

\bibitem{san2}
 M. Sango, J.L. Woukeng,  Stochastic $\sigma$-convergence and applications, Dyn. Partial
Differ. Equ. 8 (2011) 261–310.

\bibitem{sig}
 L. Signing, Two-scale convergence of unsteady Stokes type equations, SOP Trans. Appl. Math. 1(3) (2014)
23–38.

\bibitem{sig2}
 L. Signing, Periodic homogenization of the non-stationary
Navier–Stokes type equations, Afr. Mat.  28 (2017) 515–548.

\bibitem{Simon}
J. Simon,  Compact sets in the space $L^{p}(0,T;B)$, Ann. Math. Pura. Appl. 146(1) (1986) 65--96.

\bibitem{vla}
 A.A.  Vlasov,  The vibrational properties of an electron gas, Sov. Phys. Usp. 10 (1968) 721–733.


\bibitem{wang1}
W. Wang, D. Cao, J. Duan,  Effective macroscopic dynamics of stochastic partial differential equations in perforated
domains, SIAM J. Math. Anal. 38 (2007) 1508–1527.

\bibitem{wang2}
 W. Wang, J. Duan,  Homogenized dynamics of stochastic partial differential equations with dynamical boundary
conditions, Commun. Math. Phys. 275 (2007) 163–186.

\bibitem{lanzhou}
H. Yang, X. Han, C. Zhao, Existence and homogenization of trajectory statistical solutions for the 3D incompressible Hall-MHD equations, Discre. Contin. Dyn. Sys.-S (2024) doi:10.3934/dcdss.2024195.



\bibitem{zhi}
V.V. Zhikov, On the two-scale convergence, J. Math. Sci. 120  (2004) 1328–1352.










\end{thebibliography}
\end{document}